%% file: gen_d.tex
\documentclass[oneside, 11pt, reqno]{amsart}

\input{preamble}

\tikzcdset{dual/.style={cyanc, squiggly, <->}}
\tikzcdset{eqv/.style={red, <->}}

\title{Maximal $d$-spectra and locally compact Hausdorff spaces}

\date{}


\author{G. Bezhanishvili}
\address{Department of Mathematical Sciences, New Mexico State University, Las Cruces, NM, USA}

\author{P. Bhattacharjee}
\address{Department of Mathematics and Statistics, Florida Atlantic University, Boca Raton, FL, USA}

\author{S. D. Melzer}
\address{Department of Mathematics, University of Salerno, Fisciano, Italy}

\author{M. A. Moshier}
\address{Schmid College, Chapman University, Orange, CA, USA}

\begin{document}
\begingroup
\def\uppercasenonmath#1{}
\let\MakeUppercase\relax
\maketitle
\endgroup

\begin{abstract}
It is an interesting open problem whether every compact Hausdorff space can be realized as the maximal $d$-spectrum of an arithmetic frame. We approach this problem by generalizing the $d$-nucleus to a stably continuous frame. We use Priestley duality to characterize the resulting $\d$-nucleus, which allows us to prove that 
every locally compact Hausdorff space can be realized as the maximal $\d$-spectrum of a continuous regular frame. 
As a corollary, we obtain that every locally Stone space can be realized as the maximal $d$-spectrum of an algebraic regular frame. 
\end{abstract}

\setcounter{tocdepth}{1}
\tableofcontents

\section{Introduction}

An ideal $I$ of a commutative ring is a {\em $d$-ideal} provided
\[
a \in I \Longrightarrow a^{\perp\perp} \subseteq I,
\]
where $a^{\perp}$ denotes the annihilator of $a$. This concept originates in \cite{Spe72} under the name of {\em Baer ideal}, and was
further studied by numerous authors (see, e.g., \cite{Eva72,Jay84,Con88}). The term $d$-ideal was introduced by Mason \cite{Mas89}, following the usage of the term in Riesz spaces (see, e.g., \cite{Lux73,HdP80a}). 

A stronger notion is that, for each finite set $S$,
\[
S \subseteq I \Longrightarrow S^{\perp\perp} \subseteq I.
\] 
In general, these two notions are not equivalent, and the conditions under which they are equivalent have been studied in the literature (see, e.g.,  \cite{Mas89,AM11,Dub19}).  

The situation improves in 
Riesz spaces, where the two notions become equivalent for ordered ideals (also known as $\ell$-ideals). 
For each $\ell$-ideal $I$, 
there is a least $d$-ideal $dI$ containing $I$, and the correspondence $I \mapsto dI$ is a nucleus, whose fixpoints are precisely the $d$-ideals. 
Mart\'inez and Zenk \cite{MZ03} initiated the study of the corresponding nucleus on an arbitrary arithmetic frame, which they termed the {\em $d$-nucleus}. Given such a frame $L$ and denoting the pseudocomplement of $a \in L$ by $a^*$, we have:
\[
    da = \bigvee \{k^{**} \mid k \text{ is compact and } k \leq a\}.
\]
They showed that $d$ is the largest nucleus below the double-negation nucleus that preserves directed joins. In other words, using the terminology of \cite{Esc99}, $d$ is the largest {\em Scott-continuous} nucleus below the double-negation nucleus. The $d$-nucleus and the corresponding sublocale $dL$ were further studied in \cite{DI13,DS19,BKM23}. 

The study of the spectrum of maximal $d$-elements was initiated in \cite{Bha19}. More broadly, this is part of the problem of determining which spaces arise as spectra of frames. Maximal spectra provide a classical example. In particular, Wallman's representation theorem \cite{Wal38} shows that every compact $T_1$-space arises as the maximal spectrum of a suitable frame. It is therefore natural to ask for analogous representation results for maximal $d$-spectra. In general, the maximal $d$-spectrum is $T_1$ and it is compact provided $L$ has a unit \cite{Bha19}.
However, it need not be Hausdorff \cite{BBM25}. This is surprising since in the classical setting of $d$-ideals and $d$-subgroups, the corresponding space is Hausdorff \cite{HdP83,BM18}. As noted in \cite[p.~378]{Bha19}, a sufficient condition for Hausdorffness of the maximal $d$-spectrum of an arithmetic frame is \emph{disjointification}.\footnote{Banaschewski \cite{Ban97} refers to this condition as \emph{coherent normality}.} In general, it remains open which compact Hausdorff spaces arise as maximal $d$-spectra of arithmetic frames.

In this paper, we approach this problem by generalizing the $d$-nucleus to the setting of \emph{stably continuous frames}. Since arithmetic frames are precisely the algebraic stably continuous frames, this is the natural ambient category in which to seek such a generalization. 
While the classical $d$-nucleus is defined in terms of compact elements of an arithmetic frame, in the broader setting of continuous frames one works instead with the {\em way-below relation} $\ll$ (see, e.g., \cite{GH+03}).
Replacing compact elements with the way-below relation gives rise to the following definition:
\[
    \d a = \bigvee\{b^{**} \mid b \ll a\}.
\]
This is precisely Escard\'o's \emph{continuous coreflection} of the double-negation nucleus  on a stably continuous frame, which means that $\d$ is the largest Scott-continuous nucleus beneath the double-negation nucleus  \cite[Lem.~3.1.12, Thm.~3.1.15]{Esc98}.\footnote{The stably continuous frames considered in \cite{Esc98} are compact, but compactness is not used in his proofs.} 

We point out that Scott-continuous nuclei form a natural class of nuclei, which is particularly well behaved on a continuous frame $L$: while the sublocale corresponding to an arbitrary nucleus on $L$ need not be continuous, those arising from Scott-continuous nuclei always are (see, e.g., \cite[Lem.~VII-2.3]{Joh82}). In \cite{Esc01}, it was shown that the Scott-continuous nuclei form the patch frame of $L$.
In the setting of algebraic frames, they reduce to the inductive nuclei of \cite{MZ03}.

Using Priestley duality for stably continuous frames \cite{BM23}, we obtain a Priestley-theoretic description of Scott-continuous nuclei and apply it to the $\d$-nucleus, describing the corresponding nuclear subset and its space of points, thereby extending the results of \cite{BBM25}. As an application, we obtain an alternative description of the Gleason cover of a compact Hausdorff space in terms of compact $d$-elements.
We prove that every locally compact Hausdorff space is homeomorphic to the maximal $\d$-spectrum of 
a continuous regular frame. As a consequence, we obtain that every locally Stone space is homeomorphic to the maximal $d$-spectrum of an algebraic regular frame.
While this solves the problem for zero-dimensional spaces, the general representation problem for maximal $d$-spectra of arithmetic frames remains an interesting open problem,
as does the more general problem for maximal $\d$-spectra of stably continuous frames.

The paper is organized as follows. In \cref{sec:Preliminaries}, we recall various categories of continuous and algebraic frames, 
and their 
dual categories of spaces.
In \cref{sec:Priestley-duality}, we outline Priestley duality for continuous and algebraic frames.
In \cref{sec:Scott-cts-nuclei}, we give a Priestley-theoretic characterization of Scott-continuous nuclei on continuous frames, extending the characterization of inductive nuclei on arithmetic frames.
In \cref{sec:d-nucleus}, we 
concentrate on the nucleus $d$ on an arithmetic frame and its generalization $\d$ on a stably continuous frame. We  use Priestley duality to describe the nuclear subset corresponding to $\d$ and its space of points. 
We also obtain an alternative description of the Gleason cover of a compact Hausdorff space in terms of compact $d$-elements. 
Finally, in \cref{sec:d-spectra}, we introduce the maximal $\d$-spectrum and prove that every locally compact Hausdorff space is homeomorphic to $\max(\d L)$ for some continuous regular frame $L$. As a consequence, we obtain that every locally Stone space is homeomorphic to $\max(dL)$ for some algebraic regular frame $L$.
We conclude the paper with several open problems. 

\section{Preliminaries} \label{sec:Preliminaries}
In this section, we recall various categories of continuous and algebraic frames and the corresponding dual categories of locally compact and compactly based sober spaces.

A \emph{frame} is a complete lattice $L$ satisfying the \emph{frame law} 
\[a \wedge \bigvee S = \bigvee\{a \wedge s \mid s \in S\}\]
for all $a \in L$ and $S \subseteq L$. 
A \emph{frame homomorphism} is a map between frames that preserves arbitrary joins and finite meets. Let \Frm be the category of frames and frame homomorphisms. 

Let $L$ be a frame. For $a,b \in L$, we say that $a$ is {\em way below} $b$, and write $a \ll b$, provided $b \leq \bigvee S$ implies $a \leq \bigvee T$ for some finite $T \subseteq S$. Then $a$ is \emph{compact} if $a \ll a$, and $L$ is {\em compact} if its top element $1$ is compact. We call $L$  
\begin{itemize}
    \item \emph{continuous} if each $a \in L$ is a join of elements way below it;
    \item \emph{stably continuous} if in addition $\ll$ is \emph{stable} (or \emph{multiplicative}), meaning that
    \[
    a \ll b \text{ and } a \ll c \implies a \ll b \wedge c;
    \]
    \item \emph{stably compact} if it is a compact stably continuous frame.
\end{itemize}

The following basic properties of the way-below relation are well known (see, e.g., \cite[pp.~50, 56]{GH+03}) and will be used throughout.

\begin{lemma} \label{lem:properties-of-way-below}
    Let $L$ be a frame.
    \begin{enumerate}[cref=lemma]
        \item $a \ll b$ implies $a \leq b$. \label{lem:weakening-1}
        \item $a \leq b \ll c \leq d$ implies $a \ll d$. \label{lem:weakening-2}
    \end{enumerate}
    If $L$ is continuous, then
    \begin{enumerate}[cref=lemma,resume]
        \item $a \ll b$ implies $a \ll c \ll b$ for some $c \in L$. \label{lem:interpolating}
        \item $\upset a = \bigcap \{ \upset b \mid b \ll a\}$.
        \label{upset-intersection-ll}
    \end{enumerate}
\end{lemma}

We will also make use of Scott-open filters, where we recall that a filter $\filter{F}$ of a frame $L$ is \emph{Scott-open} provided whenever $S \subseteq L$ is directed and $\bigvee S \in \filter{F}$, then $S \cap \filter{F} \neq \varnothing$. The following lemma collects some known facts about Scott-open filters that will be used in what follows (see, e.g., \cite[p.~96]{GH+03} for~\tref{lem:sep-by-SFilt}, \cite[p.~222]{HM81} for~\tref{lem:round-filters}, and \cite[Proof of Lemma~1]{BB88} for~\tref{lem:upup-Scott-2}, which readily generalizes to~\tref{lem:upup-Scott}).

\begin{lemma}
    Let $L$ be a continuous frame.

    \begin{enumerate}[cref=lemma]
        \item If $a \nleq b$, then there exists a Scott-open filter $\filter{F} \subseteq L$ such that $a \in \filter{F}$ and $b \notin \filter{F}$.\label{lem:sep-by-SFilt}
        \item If $\filter{F}$ is a Scott-open filter and $a \in \filter{F}$, then there exists $b \in \filter{F}$ such that $b \ll a$.\label{lem:round-filters}
        \item If $\filter{F}$ is a Scott-open filter and $\bigvee S \in \filter{F}$, then there exists a finite $T \subseteq S$ such that $\bigvee T \in \filter{F}$.\label{lem:SFilt-finite-joins}
        \item If $L$ is stably continuous, then $\twoheaduparrow a := \{b \in L \mid a \ll b\}$ is a Scott-open filter for each $a \ll 1$.\label{lem:upup-Scott}
        \item If $L$ is stably compact, then $\twoheaduparrow a$ is a Scott-open filter for each $a \in L$. \label{lem:upup-Scott-2}
    \end{enumerate}
\end{lemma}

A frame homomorphism $h : L \to M$ is \emph{proper} if $a \ll b$ implies $ha \ll hb$. It is easy to see that the identity frame homomorphism is proper and that the composition of two proper frame homomorphisms is proper.

\begin{definition} \label{def:ConFrm}
    \leavevmode
    \begin{enumerate}
        \item Let \ConFrm be the category of continuous frames and proper frame homomorphisms.
        \item Let \StConFrm be the full subcategory of \ConFrm consisting of stably continuous frames.
        \item Let \StKFrm be the full subcategory of \StConFrm consisting of stably compact frames.
    \end{enumerate}
\end{definition}

Another significant relation in frames is the well-inside relation. We recall that $a$ is \emph{well inside} $b$, written $a \prec b$, if $a^* \vee b = 1$, where 
\[
    a^* = \bigvee\{c \in L \mid c \wedge a = 0\}
\]
is the \emph{pseudocomplement} of $a$. Then a frame is \emph{regular} provided each element is a join of elements well inside it.

\begin{definition}
    \leavevmode
    \begin{enumerate}
        \item Let $\ConRFrm$ be the full subcategory of \ConFrm consisting of continuous regular frames.
        \item Let $\KRFrm$ be the category of compact regular frames and frame homomorphisms.
    \end{enumerate}
\end{definition}

It is well known 
(see, e.g., \cite[p.~8]{BB88}) that \KRFrm is a full subcategory of \StKFrm. This is based on the fact that in a regular frame, $\ll$ is a subrelation of $\prec$ (see, e.g., \cite[p.~135]{PP12}). We thus have:

\begin{lemma} \label{conreg-is-stably}
    Continuous regular frames are stably continuous. 
\end{lemma}
\begin{proof}
Let $L$ be continuous and regular, $a \ll b,c$, and $b\wedge c \leq \bigvee S$ for some $S \subseteq L$. Since $L$ is regular, from $a \ll b,c$ it follows that $a \prec b,c$, so $a \prec b \wedge c$. Therefore, $ a^* \vee (b \wedge c) = 1$, so $a^* \vee \bigvee S = 1$, and hence $b \le a^* \vee \bigvee S$. Since $a \ll b$, there exists a finite $T \subseteq S$ with $a \le a^* \vee \bigvee T$. Thus, $a \le \bigvee T$, yielding that $a \ll b \wedge c$. Consequently, $L$ is stably continuous.
\end{proof}

It follows that \ConRFrm is the full subcategory of \StConFrm consisting of regular frames and \KRFrm is the full subcategory of \ConRFrm consisting of compact frames.

We now turn our attention to algebraic frames by replacing the way-below relation with compact elements. 
We write $\K(L)$ for the set of compact elements of a frame $L$. Then $L$ is
\begin{itemize}
    \item \emph{algebraic} if each $a \in L$ is a join from $\K(L)$;
    \item \emph{arithmetic} if in addition $\K(L)$ is closed under binary meets;
    \item \emph{coherent} if it is a compact arithmetic frame.
\end{itemize}

In algebraic frames, the way-below relation is 
determined by compact elements:

\begin{lemma}[{see, e.g., \cite[p.~116]{GH+03}}]
    Let $L$ be an algebraic frame. Then $a \ll b$ iff $a \leq k \leq b$ for some $k \in \K(L)$. \label{lem:ll-and-alg}
\end{lemma}

A frame homomorphism $h : L \to M$ is \emph{coherent} if $k \in \K(L)$ implies $hk \in \K(M)$. It is easy to see that the identity frame homomorphism is coherent and that the composition of two coherent frame homomorphisms is coherent. 
\begin{definition}
    \leavevmode
    \begin{enumerate}
        \item Let \AlgFrm be the category of algebraic frames and coherent frame homomorphisms.
        \item Let \AriFrm be the full subcategory of \AlgFrm consisting of arithmetic frames.
        \item Let \CohFrm be the full subcategory of \AriFrm consisting of coherent frames.
    \end{enumerate}
\end{definition}

In algebraic frames, stability of the way-below relation corresponds precisely to being an arithmetic frame. We thus have:
\begin{proposition}[{\cite[p.~117]{GH+03}}]
    Let $L$ be an algebraic frame.
    \begin{enumerate}
        \item $L$ is arithmetic iff $L$ is stably continuous.
        \item $L$ is coherent iff $L$ is stably compact.
    \end{enumerate}
\end{proposition}

In view of \cref{lem:ll-and-alg}, a frame homomorphism between algebraic frames is coherent iff it is proper. 
Consequently, the categories \AlgFrm, \AriFrm, and \CohFrm are full subcategories of \ConFrm, \StConFrm, and \StKFrm, respectively. 

An element $a$ of a frame is \emph{complemented} if $a \prec a$.
A frame is \emph{zero-dimensional} if each element is a join of complemented elements. It is easy to see that each zero-dimensional frame is regular.  A frame $L$ is
\begin{itemize}
    \item \emph{locally Stone} if it is a continuous zero-dimensional frame;
    \item \emph{Stone} if it is a compact zero-dimensional frame.
\end{itemize}

\begin{definition}
    \leavevmode
    \begin{enumerate}
        \item Let \LStoneFrm be the full subcategory of \ConFrm consisting of locally Stone frames.
        \item Let \StoneFrm be the full subcategory of \KRFrm consisting of Stone frames.
    \end{enumerate}
\end{definition}

Every locally Stone frame is algebraic (see, e.g., \cite[Rem.~3.12]{BK25}), and hence every proper frame homomorphism between locally Stone frames is coherent. Thus, \LStoneFrm is a full subcategory of \AriFrm. Similarly, every Stone frame is algebraic. Therefore, since Stone frames are compact regular, 
every frame homomorphism between Stone frames is coherent. 
Thus, \StoneFrm is a full subcategory of \CohFrm. In fact, 
we have:

\begin{proposition}[{\cite[Thm.~2.4(1) and Cor.~2.6]{MZ03}}]
    Let $L$ be an algebraic frame.
    \begin{enumerate}
        \item $L$ is locally Stone iff $L$ is regular.
        \item $L$ is Stone iff $L$ is compact regular.
    \end{enumerate}
\end{proposition}
Consequently, \LStoneFrm and \StoneFrm are the full subcategories of \ConRFrm and \KRFrm, respectively, consisting of algebraic frames.
We thus have the following inclusion relations between the above categories of continuous and algebraic frames, 
where \begin{tikzcd}\cat{A} \ar[r, hook] & \cat B\end{tikzcd} stands for ``$\cat A$ is a full subcategory of $\cat B$.''

\begin{figure}[H]
    \centering
    \begin{tikzcd}[column sep={1.25cm, between origins}, row sep={.85cm, between origins}]
        \ConFrm  \\
        && \AlgFrm\\
        \StConFrm \\
        && \AriFrm\\
        & \ConRFrm\\
        &&& \LStoneFrm\\
        \StKFrm \\
        && \CohFrm\\
        & \KRFrm\\
        &&& \StoneFrm
        \ar[plotd, from=3-1, to=1-1, hook']
        \ar[plotd, from=7-1, to=3-1, hook']
        \ar[plotd, from=2-3, to=1-1, hook']
        \ar[plotd, from=4-3, to=3-1, hook']
        \ar[plotd, from=4-3, to=2-3, hook', crossing over]
        \ar[plotd, from=8-3, to=7-1, hook']
        \ar[plotd, from=8-3, to=4-3, hook', crossing over]
        \ar[plotd, from=5-2, to=3-1, hook']
        \ar[plotd, from=9-2, to=7-1, hook']
        \ar[plotd, from=9-2, to=5-2, hook', crossing over]
        \ar[plotd, from=6-4, to=4-3, hook', crossing over]
        \ar[plotd, from=6-4, to=5-2, hook', crossing over]
        \ar[plotd, from=10-4, to=9-2, hook']
        \ar[plotd, from=10-4, to=8-3, hook', crossing over]
        \ar[plotd, from=10-4, to=6-4, hook', crossing over]
    \end{tikzcd}
    \caption{Categories of continuous and algebraic frames}\label{fig:con-and-alg}
\end{figure}
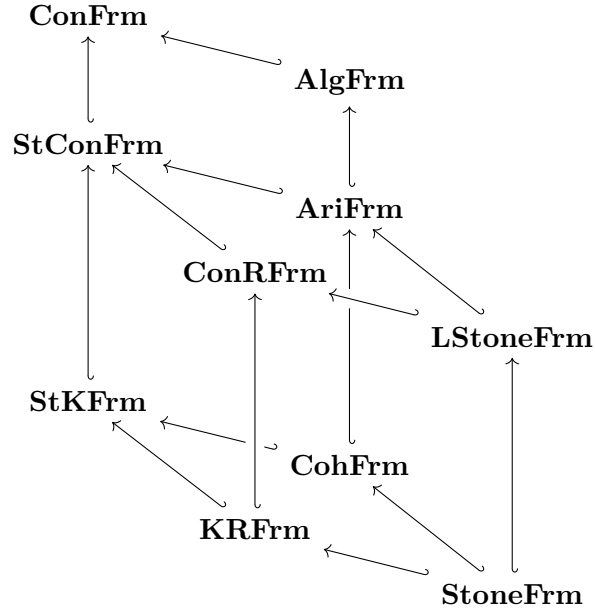

We next describe the categories of topological spaces that are dual to the above categories of frames.
Recall that a topological space $X$ is \emph{sober} if each irreducible closed set is the closure of a unique point, 
and that it is \emph{locally compact} if for each $x \in X$ and open neighborhood $U$ of $x$, there are an open $V$ and a compact $K$ such that $x \in V \subseteq K \subseteq U$. 
A set $S \subseteq X$ is \emph{saturated} if it is an intersection of open sets. 
In view of \cite[Lem.~VI-6.21]{GH+03}, we call a continuous map $f : X \to Y$ between locally compact sober spaces  \emph{proper} if $f^{-1}(K)$ is compact for each compact saturated set $K \subseteq Y$. 

\begin{definition}
    \leavevmode
    \begin{enumerate}
        \item Let $\LCSob$ be the category of locally compact sober spaces and proper maps.
        \item Let $\SLCSp$ be the full subcategory of $\LCSob$ consisting of spaces in which the binary intersection of two compact saturated sets is compact.
        \item Let $\SKSp$ be the full subcategory of $\SLCSp$ consisting of spaces that are additionally compact.
        \item Let $\LCHaus$ be the full subcategory of $\LCSob$ consisting of locally compact Hausdorff spaces. 
        \item Let $\KHaus$ be the category of compact Hausdorff spaces and continuous maps.
    \end{enumerate}
\end{definition}
Since locally compact Hausdorff spaces are stably locally compact, $\LCHaus$ is a full subcategory of $\SLCSp$. Therefore, since continuous maps between compact Hausdorff spaces are proper, $\KHaus$ is a full subcategory of both $\LCSob$ and $\SKSp$ (see, e.g., \cite[p.~66]{BH14}). The next result is well known (see, e.g., \cite[Prop.~V-5.20 and Sec.~VI-7.4]{GH+03} or \cite[Sec.~VII-4]{Joh82}).

\begin{theorem} \label{thm:confrm-duality}
    \leavevmode
    \begin{enumerate}[cref=theorem]
        \item \ConFrm is dually equivalent to \LCSob. \label{hofmann-lawson}
        \item \StConFrm is dually equivalent to \SLCSp.
        \label{hofmann-lawson-2}
        \item \StKFrm is dually equivalent to \SKSp.
        \item \ConRFrm is dually equivalent to \LCHaus. \label{hofmann-lawson-LC}
        \item \KRFrm is dually equivalent to \KHaus. \label{isbell-duality}
    \end{enumerate}
\end{theorem}
\begin{remark}
The duality in \cref{hofmann-lawson} is known as \emph{Hofmann--Lawson duality}~\cite{HL78}, and the duality in \cref{isbell-duality} as \emph{Isbell duality}~\cite{Isb72}. 
These, together with the other dual equivalences above, arise as restrictions of the well-known dual adjunction between \Frm and the category \Top of topological spaces and continuous maps (see, e.g., \cite[p.~44]{Joh82}), which we briefly recall.

The functor $\Omega : \Top \to \Frm$ maps a topological space $X$ to its frame of opens $\Omega(X)$ and a continuous map $f : X \to Y$ to the frame homomorphism $f^{-1} : \Omega(Y) \to \Omega(X)$. 
The functor $\pt : \Frm \to \Top$ maps a frame $L$ to its \emph{space of points} $\pt L$, where a \emph{point} of $L$ is a completely prime filter, and the opens of $\pt L$ are sets of the form
\[
    \varphi(a) = \{P \in \pt L \mid a \in P\}.
\]
A frame homomorphism $h : L \to M$ is mapped to the continuous map 
\[
    \pt(h) := h^{-1} : \pt M \to \pt L.
\]
An immediate consequence of Hofmann--Lawson duality is that continuous frames are \emph{spatial}, that is, isomorphic to $\Omega(X)$ for some topological space $X$.
\end{remark}

Following \cite[p.~2063]{Ern09}, we call a topological space $X$ \emph{compactly based} if it has a basis of compact opens. 
A subset $U$ of $X$ is \emph{clopen} if it is both closed and open, $X$ is \emph{zero-dimensional} if it has a basis of clopens, and $X$ is a \emph{Stone space} if it is compact, Hausdorff, and zero-dimensional. A space $X$ is \emph{locally Stone} if it is locally compact, Hausdorff, and zero-dimensional, i.e., the definition of locally Stone spaces is obtained from that of Stone spaces by replacing compactness with local compactness. 
A continuous map $f : X \to Y$ is \emph{coherent} if $f^{-1}(U)$ is compact for each compact open $U \subseteq Y$.  

\begin{definition}
    \leavevmode
    \begin{enumerate}
        \item Let \KBSob be the category of compactly based sober spaces and coherent maps.
        \item Let \SKBSp be the full subcategory of \KBSob consisting of those spaces in which the binary intersection of two compact opens is again compact.
        \item Let \Spec be the full subcategory of \SKBSp consisting of those spaces that are additionally compact.
        \item Let \LStone be the full subcategory of \LCSob consisting of locally Stone spaces.
        \item Let \Stone  be the category of Stone spaces and continuous maps.
    \end{enumerate}
\end{definition}

\begin{remark}
    \leavevmode
    \begin{enumerate}
        \item A continuous map between locally compact sober spaces is coherent iff it is proper (see, e.g., \cite[Rem.~2.9]{BM25}), and hence \KBSob is a full subcategory of \LCSob. 
        \item Since compact opens in Stone spaces are clopen, it is immediate that continuous maps between Stone spaces are coherent. 
        Thus, \Stone is a full subcategory of both \Spec and \LStone. 
        \item By a relativized argument, locally Stone spaces are stably compactly based, and hence \LStone is a full subcategory of \SKBSp.
    \end{enumerate}
\end{remark}

The dualities of \cref{thm:confrm-duality} restrict to the algebraic setting to yield the following dualities. 

\needspace{2em}
\begin{theorem} \label{thm:Alg-dualities}
    \leavevmode
    \begin{enumerate}[cref=theorem]
        \item \AlgFrm is dually equivalent to \KBSob.
        \item \AriFrm is dually equivalent to \SKBSp.
        \item \CohFrm is dually equivalent to \Spec.
        \item \LStoneFrm is dually equivalent to \LStone.\label{thm:Stone}
        \item \StoneFrm is dually equivalent to \Stone.\label{thm:LStone}
    \end{enumerate}
\end{theorem}

\begin{remark}
    The dualities for \AlgFrm, \AriFrm, and \CohFrm go back to \cite[p.~52]{HK72} (see also \cite[p.~423]{GH+03}), while the dualities for \LStoneFrm and \StoneFrm arise by restricting 
    the above to the setting of regular frames. For the latter duality see \cite{Ban89}, 
    which has its roots in the celebrated Stone duality for boolean algebras \cite{Sto36}.
\end{remark}

\cref{fig:con-and-loc} summarizes the correspondences described in this section, where we use the following conventions:
\begin{itemize}
    \item \begin{tikzcd}\cat{A} \ar[r, <->, dual, cyanc] & \cat B\end{tikzcd} stands for ``$\cat A$ is dually equivalent to $\cat B$;''
    \item \begin{tikzcd}\cat{A} \ar[r, hook] & \cat B\end{tikzcd} stands for ``$\cat A$ is a full subcategory of $\cat B$.''
\end{itemize}

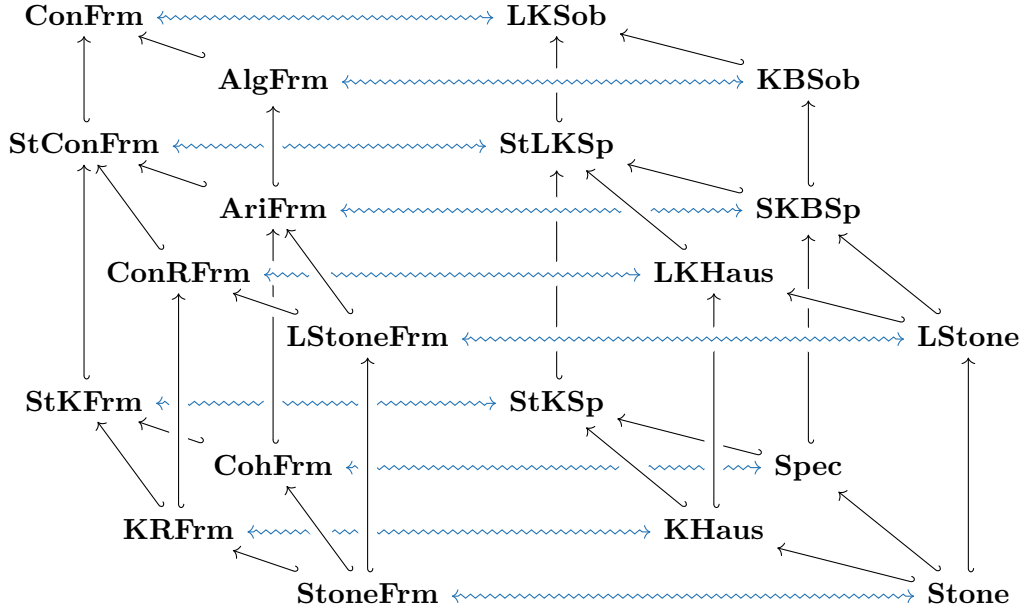
\begin{figure}[H]
    \centering
    \begin{tikzcd}[column sep={1.25cm, between origins}, row sep={.85cm, between origins}]
        \ConFrm  &&&&& \LCSob \\
        && \AlgFrm &&&&& \KBSob\\
        \StConFrm &&&&& \SLCSp \\
        && \AriFrm &&&&& \SKBSp\\
        & \ConRFrm &&&&& \LCHaus \\
        &&& \LStoneFrm &&&&& \LStone\\
        \StKFrm &&&&& \SKSp\\
        && \CohFrm &&&&& \Spec \\
        & \KRFrm &&&&& \KHaus\\
        &&& \StoneFrm &&&&& \Stone
        \ar[from=3-1, to=1-1, hook']
        \ar[from=3-6, to=1-6, hook']
        \ar[from=7-1, to=3-1, hook']
        \ar[from=7-6, to=3-6, hook']
        \ar[from=1-1, to=1-6, dual]
        \ar[from=3-1, to=3-6, dual]
        \ar[from=7-1, to=7-6, dual]
        \ar[from=2-3, to=1-1, hook']
        \ar[from=2-8, to=1-6, hook']
        \ar[from=4-3, to=3-1, hook']
        \ar[from=4-8, to=3-6, hook']
        \ar[from=4-3, to=2-3, hook', crossing over]
        \ar[from=4-8, to=2-8, hook']
        \ar[from=8-3, to=7-1, hook']
        \ar[from=8-8, to=7-6, hook']
        \ar[from=8-3, to=4-3, hook', crossing over]
        \ar[from=8-8, to=4-8, hook']
        \ar[from=2-3, to=2-8, dual, crossing over]
        \ar[from=4-3, to=4-8, dual, crossing over]
        \ar[from=8-3, to=8-8, dual]
        \ar[from=5-2, to=3-1, hook']
        \ar[from=5-7, to=3-6, hook', crossing over]
        \ar[from=9-2, to=7-1, hook']
        \ar[from=9-7, to=7-6, hook', crossing over]
        \ar[from=9-2, to=5-2, hook', crossing over]
        \ar[from=9-7, to=5-7, hook', crossing over]
        \ar[from=5-2, to=5-7, dual, crossing over]
        \ar[from=9-2, to=9-7, dual]
        \ar[from=6-4, to=4-3, hook', crossing over]
        \ar[from=6-9, to=4-8, hook', crossing over]
        \ar[from=6-4, to=5-2, hook', crossing over]
        \ar[from=6-9, to=5-7, hook', crossing over]
        \ar[from=10-4, to=9-2, hook']
        \ar[from=10-9, to=9-7, hook']
        \ar[from=10-4, to=8-3, hook', crossing over]
        \ar[from=10-9, to=8-8, hook']
        \ar[from=10-4, to=6-4, hook', crossing over]
        \ar[from=10-9, to=6-9, hook']
        \ar[from=6-4, to=6-9, dual, crossing over]
        \ar[from=10-4, to=10-9, dual]
    \end{tikzcd}
    \caption{Dual equivalences described in \cref{sec:Preliminaries}}\label{fig:con-and-loc}
\end{figure}

\section{Priestley duality for continuous and algebraic frames} \label{sec:Priestley-duality}

In this section, we recall Priestley duality for the categories of continuous and algebraic frames, as well as for their various subcategories described in the previous section.
As is customary, we call a subset $S$ of a poset $X$ an \emph{upset} if 
\[S = \upset S := \{x \in X \mid \exists s \in S : s \leq x\}\]
and a \emph{downset} if 
\[S = \downset S := \{x \in X \mid \exists s \in S : x \leq s\}.\]

\begin{definition}
    A \emph{Priestley space} is a Stone space $X$ equipped with a partial order $\leq$ satisfying \emph{Priestley separation}:   if $x,y \in X$ with $x \nleq y$ then there is a clopen upset $U$ of $X$ such that $x\in U$ and $y \notin U$.
\end{definition}

Let \Pries be the category of Priestley spaces and continuous order-preser\-ving maps. Let also \DLat be the category of bounded distributive lattices and bounded lattice homomorphisms.

\begin{theorem}[{\cite{Pri70,Pri72}}]
    \DLat is dually equivalent to \Pries.
\end{theorem}

\begin{remark}
    The dual equivalence is established by sending a bounded distributive lattice $D$ to its set of prime filters $\spec(D)$ ordered by inclusion and topologized by the basis ${\{\sigma(a) \setminus \sigma(b) \mid a,b \in D\}}$,
where 
\[\sigma(a) = \{x \in \spec(D) \mid a \in x\}.\]
A bounded lattice homomorphism $h : D \to E$ is sent to the continuous order-preserving map $h^{-1}:\spec(E) \to \spec(D)$. Conversely, a Priestley space $X$ is sent to its lattice of clopen upsets $\clopup(X)$ and a continuous order-preserving map $f : X \to Y$ 
to the bounded lattice homomorphism $f^{-1} : \clopup(Y) \to \clopup(X)$.
\end{remark}

We write $\opup(X)$ for the open upsets of a Priestley space $X$, and $\cl U$ for the topological closure of $U \subseteq X$.

\begin{definition}
\leavevmode
\begin{enumerate}
    \item An \emph{\L-space} is a Priestley space $X$ such that 
    \[
        U \in \opup(X) \implies \cl U \in \clopup(X).
    \]
    \item An \emph{\L-morphism} is a continuous order-preserving map $f : X \to Y$ between \L-spaces such that
    \[
        f^{-1}(\cl U) = \cl f^{-1}(U)
    \]
    for each $U \in \opup(Y)$.
\end{enumerate}
\end{definition}

Let \LPries be the category of \L-spaces and \L-morphisms. Priestley duality restricts to yield:

\begin{theorem}[Pultr--Sichler--Wigner duality, \cite{Wig79,PS88}]
    \Frm is dually equivalent to \LPries.
\end{theorem}

For a poset $X$, we write $\min X$ and $\max X$ for the sets of minimal and maximal points of $X$, respectively.
Throughout we will use the following basic properties of \L-spaces 
(see, e.g., \cite[Thms.~3.1.2, 3.2.1, 3.2.3]{Esa19}). 

\begin{lemma} 
    Let $X$ be an \L-space.
    \begin{enumerate}[cref=lemma]
        \item 
        If $C \subseteq X$ is closed, then both $\upset C$ and $\downset C$ are closed.\label{lem:up-down-closed}
        \item 
        If $U \subseteq X$ is \textup{(}cl\textup{)}open, then $\downset U$ is \textup{(}cl\textup{)}open
        \label{lem:Esakia-space}
        \item 
        If $F \subseteq X$ is a closed upset, then $\max F$ is closed.
        \label{lem:closed-max}
        \item If $x \in C$ and $C \subseteq X$ is closed, then
        both $\downset x \cap \min C$ and $\upset x \cap \max C$ are nonempty. 
        \label{lem:min-max-nonempty}
        \item Each closed upset \textup{(}resp.~downset\textup{)} is a directed intersection of clopen upsets \textup{(}resp. downsets\textup{)}.\label{lem:intersections-of-clopens} 
    \end{enumerate}
\end{lemma}

We view the \L-space of a frame $L$ as the bridge between $L$ and its space of points. Indeed, if $X$ is the \L-space of $L$, then $\pt L$ is realized as a distinguished subset of $X$: 

\begin{definition}
    Let $X$ be an \L-space.
    \begin{enumerate}
        \item A point $x \in X$ is called a \emph{localic point} if $\downset x$ is open.
        \item 
        The \emph{localic part} of $X$ is the set $\loc X$ of localic points of $X$.
    \end{enumerate}
\end{definition}

\begin{remark}\label{rem:iso-and-loc}
    By \cref{lem:Esakia-space}, every isolated point of an \L-space is localic, but the converse fails in general. For example, in the \L-space $X = \omega + 1$, the unique maximal point is localic but it is not isolated. However, if $x$ is a minimal point of an \L-space, 
    then $x$ is localic iff it is isolated. 
    In particular, if an \L-space happens to be an antichain, then the localic points are precisely the isolated points.
\end{remark}

If $X$ is the \L-space of a frame $L$, then $\pt L = \loc X$ (see, e.g., \cite[Prop.~2.9]{PS00}). The collection $$\{ U \cap \loc X \mid U \in \clopup(X) \}$$ is a topology on $\loc X$ such that 
the restriction of an \L-morphism $f : X \to Y$ is a well-defined continuous map $f|_{\loc X} : \loc X \to \loc Y$. Therefore, the assignment $X \mapsto \loc X$ defines a functor $\loc : \LPries \to \Top$. 
The localic part of an \L-space can be used to characterize spatiality and compactness:

\begin{theorem}
    Let $L$ be a frame and $X$ its \L-space. 
    \begin{enumerate}[cref=theorem]
        \item {\em(\cite[pp.~231--232]{PS00}; see also {\cite[Thm.~5.5]{AB+20}})}
        \label{thm:spatiality} 
        $L$ is spatial iff $\loc X$ is dense in~$X$. 
        \item {\em(\cite[Lem.~3.1]{BGJ16})} 
        \label{thm:compactness}
        $L$ is compact iff $\min X \subseteq \loc X$. 
    \end{enumerate}
\end{theorem}

In light of \cref{thm:spatiality}, we call an \L-space $X$
\begin{itemize}
    \item \emph{\L-spatial} or a \emph{spatial \L-space} if $\loc X$ is dense in $X$;
    \item \emph{\L-compact} or a \emph{compact \L-space} if $\min X \subseteq \loc X$.
\end{itemize}

\cref{thm:compactness} can be relativized to each element of $L$. Thus, $a \in L$ is compact iff ${\min \sigma(a) \subseteq \loc X}$ (see, e.g., \cite[Cor.~5.4]{BM22}). 
Closed upsets with this property correspond to  
Scott-open
filters of $L$.

\begin{definition}[{\cite[Def.~5.2]{BM22}}]\label{def:Scott-upset}
    Let $X$ be an \L-space. 
    A closed upset $F \subseteq X$ is a \emph{Scott upset} provided $\min F \subseteq \loc X$.
    We write $\SUp(X)$ for the set of Scott upsets and $\clopsup(X)$ for the set of clopen Scott upsets of $X$.
\end{definition}

Writing $\SFilt(L)$ for the set of Scott-open filters of a frame $L$, we have: 

\begin{theorem}[{\cite[Thm.~5.6]{BM22}}]\label{thm:HM-Priestley}
    Let $L$ be a frame and $X$ its \L-space. 
    Then $(\SFilt(L),\subseteq)$ is isomorphic to $(\SUp(X),\supseteq)$.
\end{theorem}

\begin{remark}
    \leavevmode
    \begin{enumerate}[cref=remark]
        \item 
        \cref{thm:HM-Priestley} is obtained by restricting the well-known isomorphism between the filters $\Filt(L)$ of $L$ and the closed upsets $\clup(X)$ of $X$ (see, e.g., \cite[p.~54]{Pri84}),
        which is given by 
        \[
            \filter{F} \in \Filt(L) \mapsto \Sigma(\filter{F}) := \bigcap\{\sigma(a) \mid a \in \filter{F} \} \in \clup(X) 
        \]
        and its inverse
        \[
            F \in \clup(X) \mapsto \Phi(F) := \{a \in L \mid F \subseteq \sigma(a)\} \in \Filt(L).
        \] \label{rem:closed-upsets-and-filters}
        
        \item In \cite{BM22}, it is shown that Scott upsets of $X$ correspond to compact saturated sets of the localic part of $X$, yielding a Priestley perspective on the Hofmann--Mislove Theorem. 
    \end{enumerate}    
\end{remark}

We next describe Priestley duals of continuous and algebraic frames. We start by characterizing the way-below relation of a frame in terms of its \L-space.

\begin{lemma}[{\cite[Prop.~3.6]{PS88}}]\label{lem:ll-char}
Let $L$ be a frame and $X$ its \L-space. Identifying $L$ with $\clopup(X)$, we have $U \ll V$ iff
\[
    \forall W \in \opup(X), \, V \subseteq \cl W \implies U \subseteq W.
\]
\end{lemma}
Let $X$ be an \L-space and $Y, Z \subseteq X$. 
Motivated by \cref{lem:ll-char}, we write $Y \ll Z$ provided $Z \subseteq \cl W$ implies $Y \subseteq W$ for each $W \in \opup(X)$. 

\begin{remark}
\leavevmode
\begin{enumerate}[cref=remark]
    \item \label{Scott-upsets-property}
    With the above notation, the Scott upsets are precisely the closed upsets $F \subseteq X$ such that $F \ll F$ 
(see, e.g., \cite[Lem.~5.1]{BM22}).
    \item \label{rem:loc-X-and-open-upsets}
    Let $x \in \loc X$. 
    \cref{Scott-upsets-property} implies that for each $U \in \opup(X)$, from $x \in \cl U$ it follows that $x \in U$. Indeed, $x \in \loc X$ implies that $\upset x \in \SUp(X)$. Therefore, since $\cl U$ is an upset 
    (see, e.g., \cite[Thm.~3.1.2(VI)]{Esa19}), 
    \[
    x \in \cl U \implies \upset x \subseteq \cl U \implies \upset x \subseteq U \implies x \in U.
    \]
\end{enumerate}

\end{remark}

\begin{definition}\leavevmode
    \begin{enumerate}
        \item An \L-space $X$ is \emph{\L-continuous} or a \emph{continuous \L-space} if the \emph{continuous kernel}
    \[
        \con U := \bigcup\{V \in \clopup(X) \mid V \ll U\}
    \]
    is dense in $U$ for each $U \in \clopup(X)$.
    \item A continuous \L-space $X$ is \emph{\L-stably continuous} or a \emph{stably continuous \L-space} if 
    \[
        \con (U \cap V) = \con U \cap \con V
    \]
    for all $U,V \in \clopup(X)$.
    \item A stably continuous \L-space $X$ is \emph{\L-stably compact} or a \emph{stably compact \L-space} if it is \L-compact.
    \item An \L-morphism $f : X \to Y$ is \emph{\L-proper} if
    \[
        f^{-1}(\con U) = \con f^{-1}(U)
    \]
    for each $U \in \opup(Y)$.
    \end{enumerate}
\end{definition}

Let \ConLPries be the category of continuous \L-spaces and proper \L-morphisms. Let also \StConLPries be the full subcategory of \ConLPries consisting of stably continuous \L-spaces, and \StKLPries the full subcategory of \StConLPries consisting of stably compact \L-spaces.

\begin{theorem}[{\cite[Sects.~5 and 6]{BM23}}] \label{thm:ConLP-dualities}
\leavevmode
    \begin{enumerate}
        \item 
        \ConLPries is equivalent to \LCSob and dually equivalent to \ConFrm.
        \item 
        \StConLPries is equivalent to \SLCSp and dually equivalent to \StConFrm.
        \item 
        \StKLPries is equivalent to \SKSp and dually equivalent to \StKFrm.
    \end{enumerate}
\end{theorem}

\begin{remark} \label{L-cts-implies-spatial}
    Since every continuous frame is spatial, as a consequence of the above we obtain that every continuous \L-space is \L-spatial. This will be used frequently in later sections.
\end{remark}

The equivalences of \cref{thm:ConLP-dualities} further restrict to the setting of continuous regular frames.

\begin{definition}
    An \L-space $X$ is \emph{\L-regular} or a \emph{regular \L-space} if the \emph{regular kernel}
    \[
        \reg U := \bigcup \{V \in \clopup(X) \mid \downset V \subseteq U\}
    \]
    is dense in $U$ for each $U \in \clopup(X)$. 
\end{definition}

Let \ConRLPries be the full subcategory of \StConLPries consisting of continuous regular \L-spaces, and let \KRLPries be the full subcategory of \LPries consisting of compact regular \L-spaces. 
Every compact regular \L-space is \L-stably compact and every \L-morphism between compact regular \L-spaces is \L-proper, so \KRLPries is a full subcategory of \StKLPries (see, e.g., \cite[Cor.~7.20]{BM23}). 
Since a frame is regular iff its \L-space is \L-regular (see, e.g., \cite[Lem.~3.6]{BGJ16}), 
\cref{thm:ConLP-dualities} restricts to yield:

\begin{theorem}[{\cite[Cors.~7.8 and 7.12]{BM23}}]
\leavevmode
    \begin{enumerate}[cref=theorem]
        \item \label{thm:ConRLPries}
        \ConRLPries is equivalent to \LCHaus and dually equivalent to \ConRFrm. 
        \item \KRLPries is equivalent to \KHaus and dually equivalent to \KRFrm. \label{thm:KRLPries}
    \end{enumerate}
\end{theorem}

We next define the \L-spaces corresponding to algebraic, arithmetic, and coherent frames. 
\begin{definition}
\leavevmode
\begin{enumerate}
    \item An \L-space $X$ is \emph{\L-algebraic} or an \emph{algebraic \L-space} if the \emph{algebraic kernel}  
    \[
        \alg U := \bigcup \{V \in \clopsup(X) \mid V \subseteq U\}
    \]
    is dense in $U$ for each $U \in\clopup(X)$.
    \item An algebraic \L-space $X$ is \emph{\L-arithmetic} or an \emph{arithmetic \L-space} if 
    \[
        \alg(U \cap V) = \alg U \cap \alg V
    \]
    for all $U,V \in \clopup(X)$.
    \item An arithmetic \L-space $X$ is \emph{\L-coherent} or a \emph{coherent \L-space} if it is \L-compact.
    \item An \L-morphism $f: X \to Y$ is \emph{coherent} if
    \[
        f^{-1}(\alg U) = \alg f^{-1}(U)
    \]
    for each $U \in \opup(Y)$.
\end{enumerate}
\end{definition}

Let \AlgLPries be the category of algebraic \L-spaces and coherent \L-mor\-phisms. Let also \AriLPries and \CohLPries be the full subcategories of \AlgLPries consisting of arithmetic \L-spaces and coherent \L-spaces, respectively. 
Then \cref{thm:ConLP-dualities}
restricts to yield  
the following theorem (see \cite{BM25},  Theorems~4.11 and 5.5 and Corollaries~4.14 and 5.9).

\begin{theorem}\label{thm:alglp-dualities}
\leavevmode
    \begin{enumerate}
        \item \AlgLPries is equivalent to \KBSob and dually equivalent to \AlgFrm.
        \item \AriLPries is equivalent to \SKBSp and dually equivalent to \AriFrm.
        \item \CohLPries is equivalent to \Spec and dually equivalent to \CohFrm.
    \end{enumerate}
\end{theorem}

These equivalences further restrict to the setting of locally Stone and Stone frames.

\begin{definition}
\leavevmode
\begin{enumerate}
    \item An \L-space $X$ is \emph{\L-zero-dimensional} or a \emph{zero-dimensional \L-space} provided for each $U \in \clopup(X)$, the \emph{zero-dimensional kernel} 
    \[
        \cen U := \bigcup \{V \in \clopup(X) \mid V = \downset V \subseteq U\}
    \]
    is dense in $U$. 
    \item A zero-dimensional \L-space $X$ is \emph{\L-locally Stone} or a \emph{locally Stone \L-space} provided it is \L-continuous.
    \item A zero-dimensional \L-space $X$ is \emph{\L-Stone} or a \emph{Stone \L-space} provided it is \L-compact.
\end{enumerate}
\end{definition}

Every Stone \L-space is \L-coherent and every \L-morphism between Stone \L-spaces is a coherent \L-morphism (see \cite[Thm.~5.14]{BM25}). 
Let \LStoneLPries be the full subcategory of \AriLPries consisting of locally Stone \L-spaces, and let \StoneLPries be the full subcategory of \CohLPries consisting of Stone \L-spaces. 
Since a frame is zero-dimensional iff its \L-space is \L-zero-dimensional (see \cite[Thm.~6.3(1)]{BGJ16}), 
\cref{thm:alglp-dualities} further restricts to yield:

\begin{theorem}[{\cite[Cor.~5.19]{BM25}}]
    \leavevmode
    \begin{enumerate}
        \item \LStoneLPries is equivalent to \LStone and dually equivalent to \LStoneFrm. \label{thm:LStoneLPries} 
        \item \StoneLPries is equivalent to \Stone and dually equivalent to \StoneFrm. \label{thm:StoneLPries}
    \end{enumerate}
\end{theorem}

These correspondences are summarized 
in \cref{fig:Priestley-dualities}, where we use the following conventions:
\begin{itemize}
    \item \begin{tikzcd}\cat{A} \ar[r, eqv] & \cat B\end{tikzcd} stands for ``$\cat A$ is equivalent to $\cat B$;''
    \item \begin{tikzcd}\cat{A} \ar[r, dual] & \cat B\end{tikzcd} stands for ``$\cat A$ is dually equivalent to $\cat B$;''
    \item \begin{tikzcd}\cat{A} \ar[r, hook] & \cat B\end{tikzcd} stands for ``$\cat A$ is a full subcategory of $\cat B$.''
\end{itemize}

\begin{figure}[H]
    \centering
    \tiny
    \begin{tikzcd}[column sep={1cm, between origins}, row sep={.85cm, between origins}]
        \ConFrm  &&&&& \ConLPries &&&&& \LCSob \\
        && \AlgFrm &&&&& \AlgLPries &&&&& \KBSob\\
        \StConFrm &&&&& \StConLPries &&&&& \SLCSp \\
        && \AriFrm &&&&& \AriLPries &&&&& \SKBSp\\
        & \ConRFrm &&&&& \ConRLPries &&&&& \LCHaus \\
        &&& \LStoneFrm &&&&& \LStoneLPries &&&&& \LStone\\
        \StKFrm &&&&& \StKLPries &&&&& \SKSp\\
        && \CohFrm &&&&& \CohLPries &&&&& \Spec \\
        & \KRFrm &&&&& \KRLPries &&&&& \KHaus\\
        &&& \StoneFrm &&&&& \StoneLPries &&&&& \Stone
        \ar[from=3-1, to=1-1, hook']
        \ar[from=3-6, to=1-6, hook']
        \ar[from=3-11, to=1-11, hook']
        \ar[from=7-1, to=3-1, hook']
        \ar[from=7-6, to=3-6, hook']
        \ar[from=7-11, to=3-11, hook']
        \ar[from=1-1, to=1-6, dual]
        \ar[from=3-1, to=3-6, dual]
        \ar[from=7-1, to=7-6, dual]
        \ar[from=1-6, to=1-11, eqv]
        \ar[from=3-6, to=3-11, eqv]
        \ar[from=7-6, to=7-11, eqv]
        \ar[from=2-3, to=1-1, hook']
        \ar[from=2-8, to=1-6, hook']
        \ar[from=2-13, to=1-11, hook']
        \ar[from=4-3, to=3-1, hook']
        \ar[from=4-8, to=3-6, hook']
        \ar[from=4-13, to=3-11, hook']
        \ar[from=4-3, to=2-3, hook', crossing over]
        \ar[from=4-8, to=2-8, hook', crossing over]
        \ar[from=4-13, to=2-13, hook']
        \ar[from=8-3, to=7-1, hook']
        \ar[from=8-8, to=7-6, hook']
        \ar[from=8-13, to=7-11, hook']
        \ar[from=8-3, to=4-3, hook', crossing over]
        \ar[from=8-8, to=4-8, hook', crossing over]
        \ar[from=8-13, to=4-13, hook']
        \ar[from=2-3, to=2-8, dual, crossing over]
        \ar[from=4-3, to=4-8, dual, crossing over]
        \ar[from=2-8, to=2-13, eqv, crossing over]
        \ar[from=4-8, to=4-13, eqv, crossing over]
        \ar[from=8-3, to=8-8, dual]
        \ar[from=8-8, to=8-13, eqv]
        \ar[from=5-2, to=3-1, hook']
        \ar[from=5-7, to=3-6, hook', crossing over]
        \ar[from=5-12, to=3-11, hook', crossing over]
        \ar[from=9-2, to=7-1, hook']
        \ar[from=9-7, to=7-6, hook', crossing over]
        \ar[from=9-12, to=7-11, hook', crossing over]
        \ar[from=9-2, to=5-2, hook', crossing over]
        \ar[from=9-7, to=5-7, hook', crossing over]
        \ar[from=9-12, to=5-12, hook', crossing over]
        \ar[from=5-2, to=5-7, dual, crossing over]
        \ar[from=5-7, to=5-12, eqv, crossing over]
        \ar[from=9-2, to=9-7, dual]
        \ar[from=9-7, to=9-12, eqv]
        \ar[from=6-4, to=4-3, hook', crossing over]
        \ar[from=6-9, to=4-8, hook', crossing over]
        \ar[from=6-14, to=4-13, hook', crossing over]
        \ar[from=6-4, to=5-2, hook', crossing over]
        \ar[from=6-9, to=5-7, hook', crossing over]
        \ar[from=6-14, to=5-12, hook', crossing over]
        \ar[from=10-4, to=9-2, hook']
        \ar[from=10-9, to=9-7, hook']
        \ar[from=10-14, to=9-12, hook']
        \ar[from=10-4, to=8-3, hook', crossing over]
        \ar[from=10-9, to=8-8, hook', crossing over]
        \ar[from=10-14, to=8-13, hook']
        \ar[from=10-4, to=6-4, hook', crossing over]
        \ar[from=10-9, to=6-9, hook', crossing over]
        \ar[from=10-14, to=6-14, hook']
        \ar[from=6-4, to=6-9, dual, crossing over]
        \ar[from=10-4, to=10-9, dual]
        \ar[from=6-9, to=6-14, eqv, crossing over]
        \ar[from=10-9, to=10-14, eqv]
    \end{tikzcd}
    \caption{Equivalences and dual equivalences described in \cref{sec:Priestley-duality}}\label{fig:Priestley-dualities}
\end{figure}
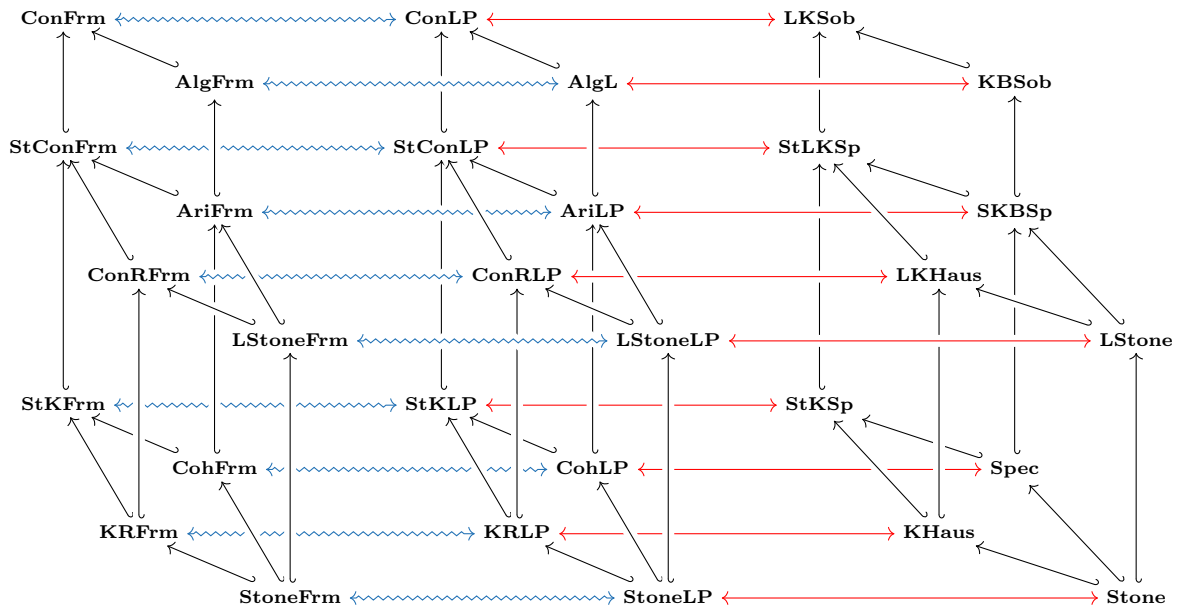

\section{Scott-continuous nuclei} \label{sec:Scott-cts-nuclei}

In this section, 
we show that Priestley duality can be used as a convenient tool to study Scott-continuous nuclei. 
In particular, we extend the characterization of inductive nuclei from \cite{BBM25} to Scott-continuous nuclei. We start by recalling that a \emph{nucleus} on a frame $L$ is a map $j : L \to L$ satisfying
\[
    j(a \wedge b) = ja \wedge jb, \quad a \leq ja, \text{ and} \quad jja \leq ja.
\]
It is well known that nuclei on $L$ are in one-to-one correspondence with sublocales of $L$ (see, e.g., \cite[p.~31]{PP12}). To recall the latter, for $a,b \in L$, let $$a \to b = \bigvee\{c \in L \mid a \wedge c \leq b\}$$ be the {\em pseudocomplement of $a$ relative to $b$}. A subset $S$ of $L$ is a \emph{sublocale} provided $S$ is closed under arbitrary meets and $a \to s \in S$ for each $a \in L$ and $s \in S$. 
For each nucleus $j : L \to L$, its image $jL = \{ja \mid a \in L\}$ is a sublocale of $L$, and conversely, each sublocale $S$ is the image of the nucleus given by $j_S a = \bigwedge (S \cap \upset a)$ for each $a \in L$. 

Let $\N(L)$ be the poset of nuclei on $L$ ordered pointwise, and let
$\S(L)$ be the poset of sublocales of $L$ ordered by inclusion.
It is well known (see, e.g., \cite[p.~51]{Joh82}) that $\N(L)$ is a frame, and that $\S(L)$ is a coframe (that is, the order dual of a frame; see, e.g., \cite[p.~28]{PP12}). 

\begin{theorem}
    For a frame $L$, the assignment $j \mapsto jL$ is a dual isomorphism between $\N(L)$ and $\S(L)$ whose inverse is given by 
    $S \mapsto j_S$.
\end{theorem}

A well-known example of a nucleus is the 
\emph{double-negation nucleus} $a \mapsto a^{**}$.
Its corresponding sublocale $\B(L)$ is known as the \emph{booleanization of $L$} (see, e.g., \cite{BP96}). 
A nucleus $j \in \N(L)$ is \emph{dense} if $j0 = 0$, and a sublocale $S \in \S(L)$ is \emph{dense} if $0 \in S$. The following is well known (see, e.g., \cite[p.~246]{PP21}): 
\begin{theorem}[Isbell's Density Theorem]  \label{thm:Isbell-density}
    For each frame $L$, the double-negation nucleus is the largest dense nucleus, and hence the booleanization $\B(L)$ is the smallest dense sublocale of~$L$.
\end{theorem}  

We are particularly interested in nuclei that behave well with respect to the way-below relation:

\begin{definition}
    A nucleus $j$ on a frame $L$ is \emph{Scott-continuous} provided $j(\bigvee S) = \bigvee j[S]$ for each directed $S \subseteq L$.
\end{definition}

\begin{remark}
    These nuclei were studied in \cite{Ban88,Sun88} under the name of \emph{finitary} nuclei.
    The term \emph{Scott-continuous nucleus} is due to Escard\'o \cite{Esc99}, who also referred to them as \emph{perfect} nuclei. 
    The Scott-continuous nuclei form a subframe of $\N(L)$,
    called the \emph{patch frame} of $L$, which plays the role of the patch topology
    in the pointfree setting (see, e.g., \cite{Esc01}).
\end{remark}

It is easy to see that a nucleus is Scott continuous iff its corresponding sublocale is closed under directed joins (see, e.g., \cite[p.~81]{GH+03}). In the setting of continuous frames, Scott-continuous nuclei are precisely those determined by the way-below relation. 

\begin{lemma} 
\label{lem:finitary-iff-inductive}
    Let $L$ be a continuous frame. For $j \in \N(L)$, the following are equivalent.
    \begin{enumerate}[cref=lemma]
        \item\label{lem:finitary}
        $j$ is Scott continuous;
        \item \label{lem:preinductive}
        $j(a) = \bigvee\{jb \mid b \ll a\}$ for each $a \in L$;
        \item \label{lem:finitary-and-filters} If $\filter{F} \in \SFilt(L)$, then $j^{-1}[\filter{F}] \in \SFilt(L)$.
    \end{enumerate}
    The following condition implies the above conditions:
    \begin{enumerate}[cref=lemma, resume]
        \item $j(a) = \bigvee\{jb \mid b \in \K(L),\ b \leq a\}$ for each $a \in L$. \label{lem:cts-inductive}
    \end{enumerate}
    If $L$ is algebraic, then all the above conditions are equivalent.
\end{lemma}
\begin{proof}
    The equivalence of \tref{lem:finitary} and \tref{lem:preinductive} (and that they are equivalent to \tref{lem:cts-inductive} when $L$ is algebraic) is shown in \cite[Props.~II-2.1]{GH+03}. It is also clear that \tref{lem:cts-inductive} implies \tref{lem:preinductive}. Thus, it suffices to show that \tref{lem:finitary} and \tref{lem:finitary-and-filters} are equivalent. 
 
    \tref{lem:finitary}$\Rightarrow$\tref{lem:finitary-and-filters} Suppose $\filter{F} \in \SFilt(L)$. Since $j$ is a nucleus, $j^{-1}[\filter{F}]$ is a filter. To see that it is Scott open, let $S \subseteq L$ be directed and $\bigvee S \in j^{-1}[\filter{F}]$. By \tref{lem:finitary}, $\bigvee j[S] = j(\bigvee S) \in \filter{F}$. Since $\filter{F} \in \SFilt(L)$ and $j[S]$ is directed, there is $s \in S$ with $js \in \filter{F}$. Therefore, $s \in j^{-1}[\filter{F}]$, and so $j^{-1}[\filter{F}] \in \SFilt(L)$.

    \tref{lem:finitary-and-filters}$\Rightarrow$\tref{lem:finitary} Let $S \subseteq L$ be directed. It suffices to show that $j(\bigvee S) \leq \bigvee j[S]$. Suppose not. Since $L$ is continuous, by \cref{lem:sep-by-SFilt} there is $\filter{F} \in \SFilt(L)$ such that $j(\bigvee S) \in \filter{F}$ and $\bigvee j[S] \notin \filter{F}$.
    By \tref{lem:finitary-and-filters}, $j^{-1}[\filter{F}] \in \SFilt(L)$, so since $\bigvee S \in j^{-1}[\filter{F}]$, there is $s \in S$ with $s \in j^{-1}[\filter{F}]$. Therefore, $js \in \filter{F}$, contradicting
    $\bigvee j[S] \notin \filter{F}$.
\end{proof}

\begin{remark}
    \label{rem:finitary-iff-inductive}
    The nuclei satisfying 
    \cref{lem:cts-inductive} 
    are called \emph{inductive} 
    in \cite{MZ03}. 
    Thus, a nucleus on an algebraic frame is Scott continuous iff it is inductive.
\end{remark}

One important feature of Scott-continuous nuclei on a continuous frame is that they preserve the way-below relation:

\begin{lemma}[{see, e.g., \cite[p.~81]{GH+03}}] \label{lem:Sct-cts-preserves-ll}
    Let $L$ be continuous and $j \in \N(L)$ Scott-continuous. Then $a \ll b$ implies $ja \ll jb$. Consequently, $a \in \K(L)$ implies $ja \in \K(jL)$.
\end{lemma}

We say that a category $\cat{A}$ of frames is \emph{closed under Scott-continuous nuclei} provided that $jL \in \cat{A}$ for each $L \in \cat{A}$ and Scott-continuous $j \in \N(L)$. As a consequence of \cref{lem:Sct-cts-preserves-ll}, we obtain:

\begin{theorem}\label{closure-under-Scott-cts}
    The categories in \cref{fig:con-and-alg} are closed under Scott-continuous nuclei.
\end{theorem}

\begin{remark}
    That \ConFrm is closed under Scott-continuous nuclei goes back to \cite[Lem.~VII-2.3]{Joh82}. 
    Since nuclei preserve finite meets, it follows that \StConFrm is closed under Scott-continuous nuclei (see \cite[Thm.~3.1.7]{Esc98}). 
    It was observed in \cite[Lem.~2]{Ban88} and \cite[Lem.~2.1]{Sun88} that Scott-continuous nuclei preserve compactness, and hence \StKFrm is closed under Scott-continuous nuclei. 
    Since nuclei also preserve regularity (see, e.g., \cite[Cor.~V-5.8]{PP12}), the same holds for \ConRFrm and \KRFrm.
    Finally, 
    since Scott-continuous nuclei preserve 
    compact elements, 
    the algebraic categories appearing in \cref{fig:con-and-alg} are  closed under Scott-continuous nuclei as well (that \AriFrm is closed under Scot-continuous nuclei was first observed in \cite[Thm.~3.1.7]{Esc98}).
\end{remark}

A Priestley characterization of inductive nuclei was given in \cite{BBM25}. The remainder of this section is devoted to generalizing this result to Scott-continuous nuclei on continuous frames. We recall that nuclei on a frame correspond to certain closed subsets of the Priestley dual.

\begin{definition}
    Let $X$ be an \L-space. A closed set $N \subseteq X$ is \emph{nuclear} provided $\downset (U \cap N)$ is clopen for each clopen $U \subseteq X$. 
\end{definition}

Let $\N(X)$ be the poset of nuclear subsets of $X$ ordered by inclusion. Then $\N(X)$ is a coframe and we have: 

\begin{theorem}[{\cite[Thm.~30]{BG07}; see also \cite[Thm.~4.7 and Rem.~4.8]{BM26}}]\label{thm:nuclei-and-sublocales}
Let $L$ be a frame and $X$ its \L-space. There is a dual isomorphism 
$\alpha : \N(L) \to \N(X)$ defined by
\[
\alpha(j)=\{x\in X \mid j^{-1}[x]\subseteq x\},
\] 
whose inverse is given by 
\[
\alpha^{-1}(N)(a)=\bigvee\{b\in L \mid \sigma(b)\cap N\subseteq \sigma(a)\}.
\]
\end{theorem}

\begin{remark}\label{rem:nuclear-and-priestley}
Let $L$ be a frame, $X$ the \L-space of  $L$, and $j \in \N(L)$. To simplify notation, we write $X_j$ for $\alpha(j)$. By \cite[Lem.~25]{BG07}, $X_j$ is order-homeomorphic to the \L-space of $jL$. 
Therefore, the localic part $\loc X_j$ corresponds to the space of points of $jL$. 
By \cite[Prop.~4.9(1)]{BBM25}, $\loc X_j = X_j \cap \loc X$.
\end{remark}

Recall that a subset $S$ of a poset $X$ is \emph{cofinal} if $X = \downset S$.
If $X$ is a Priestley space, then $S$ is cofinal iff $\max X \subseteq S$ (see \cref{lem:min-max-nonempty}).
We have the following 
characterization of dense nuclei, which will be used later on. 

\begin{theorem}[{\cite[Thms.~23(2) and 28(2)]{BG07}}]
    Let $L$ be a frame, $X$ its \L-space, and ${j \in \N(L)}$.
       Then $j$ is dense iff $X_j$ is cofinal.
        \label{thm:dense}
\end{theorem}

\begin{remark} \label{rem:Isbell-density-dually}
\cref{thm:dense} paves the way for the Priestley approach to Isbell's Density Theorem: 
$\max X$ is a cofinal nuclear subset of $X$, which is the Priestley dual of the double-negation nucleus. Thus, $\max X$ is the $\L$-space of the booleanization $\B(L)$. Clearly, $\max X$ is the least among all cofinal nuclear subsets of $X$, yielding Isbell's Density Theorem (see \cite[Cor.~3.18]{BBM25}).
\end{remark}

The following lemma gives a convenient characterization of the order on $jL$ in terms of its corresponding nuclear subset.

\begin{lemma}[{\cite[Lem.~3.10(3)]{BBM25}}] \label{lem:nuc-leq}
    Let $L$ be a frame, $X$ its \L-space, $j \in \N(L)$, and $a,b \in L$. Then $ja \leq jb$ iff $\sigma(a) \cap X_j \subseteq \sigma(b)$.
\end{lemma}

We are ready to give a Priestley characterization of Scott-continuous nuclei on continuous frames.

\begin{theorem} \label{prop:nuclei-continuous}
    Let $L$ be a continuous frame, $X$ its \L-space, and $j \in \N(L)$. Then
    $j$ is Scott continuous iff $\upset (F \cap X_j) \in \SUp(X)$ for each  $F \in \SUp(X)$. 
\end{theorem}

\begin{proof}
    Let $F \in \clup(X)$. Since $X_j$ is closed, $\upset (F \cap X_j)$ is a closed upset. 
    Recalling the isomorphism $\Phi: \clup(X) \to \Filt(L)$ (see \cref{rem:closed-upsets-and-filters}), we first show that 
    \begin{equation}
        \Phi(\upset (F \cap X_j)) = j^{-1}[\Phi(F)]. \label{eqB}
    \end{equation}
    Let $a \in L$. 
    Then
    \begin{align*}
        a \in \Phi(\upset (F \cap X_j)) &\iff \upset (F \cap X_j) \subseteq \sigma(a)\\
        &\iff F \cap X_j \subseteq \sigma(a) &&\text{since $\sigma(a)$ is an upset}\\
        &\iff \exists b \in \Phi(F) : \sigma(b) \cap X_j \subseteq \sigma(a) &&\text{by compactness and \cref{lem:intersections-of-clopens}} \\
        &\iff \exists b \in \Phi(F) :  jb \leq ja &&\text{by \cref{lem:nuc-leq}} \\
        &\iff \exists b \in \Phi(F) :  b \leq ja\\
        &\iff ja \in \Phi(F) &&\text{since $\Phi(F)$ is a filter}\\
        &\iff a \in j^{-1}[\Phi(F)].
    \end{align*}

    ($\Rightarrow$) Let $j$ be Scott continuous and $F \in \SUp(X)$. Then $\Phi(F) \in \SFilt(L)$ 
    by \cref{rem:closed-upsets-and-filters}.
    Therefore, $j^{-1}[\Phi(F)] \in \SFilt(L)$ by \cref{lem:finitary-and-filters}. But $j^{-1}[\Phi(F)] = \Phi(\upset F \cap X_j))$ by \eqref{eqB}, so  $\upset (F \cap X_j) \in \SUp(X)$ by 
    \cref{rem:closed-upsets-and-filters}.
    
    ($\Leftarrow$) By \cref{lem:finitary-and-filters}, it suffices to show that $j^{-1}[\filter{F}] \in \SFilt(L)$ for each $\filter{F} \in \SFilt(L)$. Let $\filter{F} \in \SFilt(L)$. Then $\Sigma(\filter{F}) \in \SUp(X)$ 
    by \cref{rem:closed-upsets-and-filters}, 
    and hence $\upset(\Sigma(\filter{F}) \cap X_j)\in \SUp(X)$ by assumption. 
    Thus, 
    by \cref{rem:closed-upsets-and-filters} and \eqref{eqB}, 
    \[
    j^{-1}[\filter{F}]  = 
    j^{-1}[\Phi(\Sigma(\filter{F}))] =  \Phi(\upset(\Sigma(\filter{F}) \cap X_j)) \in \SFilt(L). \qedhere
    \]
\end{proof}

Putting \cref{rem:finitary-iff-inductive,prop:nuclei-continuous} together yields:

\begin{corollary}[{\cite[Thm.~4.6]{BBM25}}]
    Let $L$ be an algebraic frame, $X$ its \L-space, and ${j \in \N(L)}$. Then $j$ is inductive iff $\upset(F \cap X_j) \in \SUp(X)$ for each $F \in \SUp(X)$.
\end{corollary}

\section{The \texorpdfstring{$\d$}{d}-nucleus on a stably continuous frame} \label{sec:d-nucleus}
As we pointed out in the introduction, the $d$-nucleus was introduced by Mart\'inez and Zenk \cite{MZ03} as an abstraction of the theory of $d$-ideals in Riesz spaces (see, e.g., \cite{HdP80a,HdP80b,dP81,HdP83}). 
For an arithmetic frame, the $d$-nucleus is the largest dense inductive nucleus \cite[Rem.~5.1(iv)]{MZ03}, and hence the largest Scott-continuous nucleus beneath the double-negation nucleus (see  \cref{rem:finitary-iff-inductive}). 
Earlier, Escard\'o \cite{Esc98} studied  continuous coreflections of nuclei. Applying his technique to the double-negation nucleus on a stably continuous frame (see \cite[Thm.~3.1.15]{Esc98}) gives rise to the $\d$-nucleus, which is a natural generalization of the $d$-nucleus of Mart\'inez and Zenk to the setting of stably continuous frames.  

In this section, we recall the relevant properties of both $d$ and $\d$, and then use Priestley duality to characterize the nuclear subset corresponding to $\d$ and its space of points. We also provide a new description of the Gleason cover of a compact Hausdorff space $X$ in terms of compact $d$-elements of the frame of ideals of $\Omega(X)$.

\begin{definition}[{\cite[Def.~5.1]{MZ03}}]
For a frame $L$, define $d : L \to L$ by
    \[
    da = \bigvee\{k^{**} \mid k \in \K(L) \text{ and } k \leq a\}.
    \]
Elements in the image $dL$ are called \emph{$d$-elements}.
\end{definition}

\begin{proposition}[{\cite[Rem.~5.1]{MZ03}}]
    Let $L$ be a frame.
    \begin{enumerate}[cref=proposition]
        \item If $L$ is algebraic, then $d$ is a closure operator on $L$.
        \item\label{prop:arithmetic-makes-d-nucleus}\label{thm:d-largest-inductive}
        If $L$ is arithmetic, then $d$ is a nucleus on $L$, which is 
    the largest Scott-continuous nucleus beneath the double-negation nucleus.
    \end{enumerate}
\end{proposition}

\begin{remark}
    \label{rem:KdL=dKL} It is immediate from the definition that $dk = k^{**}$ for each $k \in \K(L)$ and that $da \leq a^{**}$ for each $a \in L$. 
In addition, if $L$ is an arithmetic frame, then $\K(dL) = d[\K(L)]$ (see \cite[Lem.~4.2]{MZ03}). This will be used in \cref{thm:Gleason}.
\end{remark} 

We next generalize $d$ by replacing compact elements with the way-below relation.

\begin{definition}
    Let $L$ be a frame. Define $\d : L \to L$ by $$\d a = \bigvee\{b^{**} \mid b \ll a\}.$$    
\end{definition}

As we pointed out in the introduction, if $L$ is stably continuous then $\d$ is precisely the continuous coreflection of the double-negation nucleus studied by Escard\'o \cite[Thm.~3.1.15]{Esc98}. 
We collect the basic properties of $\d$ in the following proposition.

\begin{proposition}\label{thm:d-on-stably-cont-L}
    Let $L$ be a frame.
    \begin{enumerate}[cref=proposition]
        \item
        \label{c-4}\label{d-order-pres} $\d$ is monotone and
        $
            da \leq \d a \leq a^{**}
        $ 
        for each $a \in L$.
        
        \item If $L$ is continuous, then $\d$ is a closure operator on $L$.
        \label{c-1}
        
        \item If $L$ is stably continuous, then $\d$ is the largest Scott-continuous nucleus below the double-negation nucleus.
        \label{c-2}\label{largest-finitary}
        
        \item If $L$ is algebraic, then $\d=d$.
        \label{c-3}
    \end{enumerate}
\end{proposition}
\begin{proof}
    \tref{c-4} That $\d$ is monotone is straightforward. If $k \in \K(L)$ with $k \leq a$ then $k \ll a$, so $k^{**} \leq \d a$. Therefore, $d a \leq \d a$.
    Moreover, if $b \ll a$ then $b \leq a$ (see \cref{lem:weakening-1}). Thus, $b^{**} \leq a^{**}$, and hence $\d a \leq a^{**}$.
    
    \tref{c-1} See \cite[Lem.~3.1.13]{Esc98}. 
    
    \tref{c-2} See \cite[Thm.~3.1.15]{Esc98}.

    \tref{c-3} By \tref{c-4}, it suffices to show that $\d a \leq da$. Since $L$ is algebraic, $b \ll a$ iff there is $k \in \K(L)$ with $b \leq k \leq a$ (see \cref{lem:ll-and-alg}).
    Hence, if $b \ll a$ then $b^{**} \leq k^{**} \leq da$. Thus, $\d a \leq da$.
\end{proof}

We now give 
a dual description of the $\d$-nucleus. Let $X$ be the \L-space of a frame $L$. It follows from \cite[Prop.~2]{Wig79} that
\[
    \sigma\!\left(\bigvee S\right) = \cl \bigcup \{\sigma(s) \mid s \in S\}
\]
for each $S \subseteq L$. 
Therefore, identifying a stably continuous frame $L$ with $\clopup(X)$, we have:
\[
    \d U = \cl \bigcup\{V^{**} \mid V \in \clopup(X),\, V \ll U\}.
\]
We first provide several alternative descriptions of $\d U$. 
For this, we recall the following lemma.

\begin{lemma}[{\cite[Prop.~5.2]{PS00}; see also \cite[Lem.~5.7]{BM23}}]
    Let $X$ be a continuous \L-space and $U,V \in \clopup(X)$. Then
        $V \ll U$ iff there exists $F \in \SUp(X)$ such that $V \subseteq F \subseteq U$. \label{lem:Scott-inbetween}
\end{lemma}

    Let $X$ be an \L-space. 
    For $U \in \clopup(X)$, the pseudocomplement is given by $U^* = X \setminus \downset U$ (see, e.g., \cite[p.~20]{Esa19}), and hence 
    $U^{**} = X \setminus \downset(X \setminus \downset U)$. 
    It will be convenient to use this notation for arbitrary subsets $S \subseteq X$, so we write 
    $S^{**} := X \setminus \downset(X \setminus \downset S)$. 
    Observe that
    $x \in S^{**}$ iff $\upset x \subseteq \downset S$.

\begin{proposition}\label{prop:con-negneg}
    Let $X$ be a continuous \L-space and $U \in \clopup(X)$. Then
    \[
        \bigcup \{V^{**} \mid V \in \clopup(X),\, V \ll U\} = (\con U)^{**}
        = \bigcup \{F^{**} \mid F \in \SUp(X),\, F \subseteq U\}.
    \]
    Consequently,
    \begin{align*}
        \d U = \cl ((\con U)^{**}) = \cl \bigcup \{F^{**} \mid F \in \SUp(X),\, F \subseteq U\}.
    \end{align*}
\end{proposition}

\begin{proof}
    First, let $x \in \bigcup \{V^{**} \mid V \in \clopup(X),\, V \ll U\}$. Then there is $V \in \clopup(X)$ such that $x \in V^{**}$ and $V \ll U$. Thus, $\upset x \subseteq \downset V \subseteq \downset \con U$, and so $x \in (\con U)^{**}$.

    Next, let $x \in (\con U)^{**}$, 
    so $\upset x \subseteq \downset \con U$, and hence $\max \upset x \subseteq \con U$. 
    Because 
    $\max \upset x$ is closed (see \cref{lem:closed-max}) and so compact, there is 
    a clopen upset
    $V \ll U$ such that $\max \upset x \subseteq V$. By \cref{lem:Scott-inbetween}, there is $F \in \SUp(X)$ such that $\max \upset x \subseteq V \subseteq F \subseteq U$. Thus, $x \in V^{**} \subseteq F^{**}$, and so $x \in \bigcup \{F^{**} \mid F \in \SUp(X),\, F \subseteq U\}$.

    Finally, let $x \in \bigcup \{F^{**} \mid F \in \SUp(X),\, F \subseteq U\}$. Then
    $x \in F^{**}$, so
    $\max \upset x \subseteq F \subseteq U$ for some $F \in \SUp(X)$. Since $X$ is a continuous \L-space, $F \subseteq \cl \con U$, so  ${F \subseteq \con U}$ because $F$ is a Scott upset 
    (see \cref{Scott-upsets-property}).
    Hence, $\max \upset x \subseteq \con U$. By compactness, there is ${V \in \clopup(X)}$ such that $\max \upset x \subseteq V \ll U$. Consequently, $x \in \bigcup \{V^{**} \mid {V \in \clopup(X),\, V \ll U}\}$.
\end{proof}

\begin{theorem} \label{thm:char-of-dU}
    Let $X$ be a stably continuous \L-space and $U \in\clopup(X)$. For $x \in X$, the following are equivalent.
    \begin{enumerate}[cref=theorem]
        \item $x \in \d U$.\label{new-in-dU-1}
        \item $x \in \cl((\con U)^{**})$\label{new-in-dU-2}
        \item If $V$ is a clopen neighborhood of $x$, then there is $y \in V$ with $\max \upset y \subseteq \con U$.\label{new-in-dU-3}
        \item If $D$ is a clopen downset neighborhood of $x$, then\label{new-in-dU-4}
        \[
            D \not\subseteq \downset (\max X \setminus \con U).
        \]
    \end{enumerate}
\end{theorem}

\begin{proof} 
    \tref{new-in-dU-1}$\Leftrightarrow$\tref{new-in-dU-2} This follows from \cref{prop:con-negneg}.

    \tref{new-in-dU-2}$\Rightarrow$\tref{new-in-dU-3} 
    Let 
    $V$ be a clopen neighborhood of $x$. By \tref{new-in-dU-2}, ${V \cap (\con U)^{**} \neq \varnothing}$.
    Since $X$ is \L-spatial (see \cref{L-cts-implies-spatial}),
    $V \cap (\con U)^{**} \cap \loc X \neq \varnothing$. Thus, there is ${y \in V \cap \loc X}$ with $y \in (\con U)^{**}$, yielding that $\max \upset y \subseteq \con U$. 

    \tref{new-in-dU-3}$\Rightarrow$\tref{new-in-dU-4} 
    By \tref{new-in-dU-3}, there is $y \in D$ such that $\max \upset y \subseteq \con U$. Therefore, $y \notin \downset (\max X \setminus \con U)$, and hence $D \not\subseteq \downset (\max X \setminus \con U)$.

    \tref{new-in-dU-4}$\Rightarrow$\tref{new-in-dU-2} Let $V$ be a clopen neighborhood of $x$ and set $D = \downset V$. By \cref{lem:Esakia-space}, $D$ is a clopen downset neighborhood of $x$. Hence, by \tref{new-in-dU-4}, $D \not\subseteq \downset (\max X \setminus \con U)$. Thus, there is $y \in D$ such that $y \not \in \downset(\max X \setminus \con U)$. 
    Therefore, $\upset y \subseteq \downset \con U$, and so $V \cap \downset \con U \neq \varnothing$. Thus, $V \cap (\con U)^{**} \neq \varnothing$, and hence $x \in \cl ((\con U)^{**})$
\end{proof}

We use \cref{thm:char-of-dU} to characterize the nuclear subset $X_{\d}$ corresponding to the $\d$-nucleus. 
By \cref{thm:nuclei-and-sublocales}, $x \in X_{\d}$ iff $\d^{-1}[x] \subseteq x$, that is,
\[
    \forall a \in L,\ \d a \in x \implies a \in x.
\]
Identifying $L$ with $\clopup(X)$, this is equivalent to
\[
    \forall U \in \clopup(X),\ x \in \d U \implies x \in U.
\]

\begin{theorem} \label{cor:char-of-Xd}
    Let $X$ be a stably continuous \L-space. For $x \in X$, the following are equivalent.
    \begin{enumerate}[cref=theorem]
        \item $x \in X_{\d}$.\label{new-Xd-1}
        \item $x \notin U \in \clopup(X)$ implies that there is a clopen downset neighborhood $D$ of $x$ such that $D \subseteq \downset(\max X \setminus \con U)$.\label{new-Xd-2}
        \item If $U \in \clopup(X)$ and for each clopen downset $D$ containing $x$ there is $y \in D \cap \loc X$ such that $\max \upset y \subseteq \con U$, then $x \in U$. \label{new-Xd-3}\label{cor:Xd-char-4}
    \end{enumerate}
\end{theorem}

\begin{proof}
    \tref{new-Xd-1}$\Rightarrow$\tref{new-Xd-2} Let $x \notin U$. Then $x \notin \d U$ by \tref{new-Xd-1}. By \cref{new-in-dU-4}, there is a clopen downset neighborhood $D$ of $x$ such that $D \subseteq \downset (\max X \setminus \con U)$.

    \tref{new-Xd-2}$\Rightarrow$\tref{new-Xd-3} Let $x\notin U$. By \tref{new-Xd-2}, there is a clopen downset neighborhood $D$ of $x$ such that $D \subseteq \downset(\max X \setminus \con U)$. Since $X$ is \L-spatial, $D \cap \loc X \ne \varnothing$. 
    Thus, there is $y \in D \cap \loc X$ such that $y \in \downset(\max X \setminus\con U)$, and hence
    $\max \upset y \not \subseteq \con U$ 
    by \cref{lem:min-max-nonempty}.

    \tref{new-Xd-3}$\Rightarrow$\tref{new-Xd-1} Let $x \in \d U$ and let $D$ be a clopen downset containing $x$. By \cref{new-in-dU-4}, ${D \not \subseteq \downset (\max X \setminus \con U)}$. 
    Since $\max X$ is closed, $\con U$ is open, and $X$ is \L-spatial,
    \[{D \cap \loc X \not\subseteq \downset(\max X \setminus \con U)}.\]
    Therefore, there is $y \in D \cap \loc X$ such that 
    $\max \upset y \subseteq \con U$. Thus, $x \in U$ by \tref{new-Xd-3}. Consequently, $x \in X_{\d}$.
\end{proof}

Observing that in an algebraic \L-space $X$, $\con U = \alg U$ for each $U \in \clopup(X)$ (see, e.g., \cite[Lem.~12.4(2)]{Mel25}), we obtain the following 
versions
of
\cref{thm:char-of-dU,cor:char-of-Xd} 
for arithmetic \L-spaces:

\begin{corollary} \label{thm:dU-char}
    Let $X$ be an arithmetic \L-space and $U \in \clopup(X)$. For $x \in X$, the following are equivalent.
    \begin{enumerate}[cref=theorem]
        \item\label{in-dU-1}
        $x \in d U$.
        \item\label{in-dU-2}
        $x \in \cl ((\alg U)^{**})$.
        \item\label{in-dU-3}
        If $V$ is a clopen neighborhood of $x$, then there is $y \in V$ such that $\max \upset y \subseteq \alg U$.
        \item\label{in-dU-4}
        If $D \subseteq X$ is a clopen downset neighborhood of $x$, then $$D \not \subseteq \downset (\max X \setminus \alg U).$$
    \end{enumerate}
\end{corollary}

\begin{corollary}  \label{cor:Xd-char}
     Let $X$ be an arithmetic \L-space. For $x \in X$, the following are equivalent.
     \begin{enumerate}[cref=corollary]
         \item $x \in X_d$. \label{Xd-1}
         \item $x \notin U \in \clopup(X)$ implies that there is a clopen downset neighborhood $D$ of $x$ such that $D  \subseteq \downset(\max X \setminus \alg U)$.\label{Xd-2}
         \item If $U \in \clopup(X)$ and for each clopen downset $D$ containing $x$ there is $y \in D \cap \loc X$ such that $\max \upset y \subseteq \alg U$, then $x \in U$.\label{Xd-3}
     \end{enumerate}
\end{corollary}

We next characterize $\loc X_{\d}$, generalizing a similar characterization of $\loc X_d$ for arithmetic \L-spaces given in \cite[Thm.~4.12]{BBM25}.
For this, we require the following lemmas, the first of which generalizes \cite[Lem.~5.7]{BM23} (see also \cite[Prop.~5.2]{PS00}).

\begin{lemma} \label{lem:con-Scott}
    Let $X$ be a continuous \L-space, $U \in \clopup(X)$, and $C$ a closed subset of $X$. Then $C \subseteq \con U$ iff there is $F \in \SUp(X)$ such that $C \subseteq F \subseteq U$.
\end{lemma}

\begin{proof}
    ($\Rightarrow$) Suppose $C \subseteq \con U$. Since $C$ is closed, it is compact, so there is $V \in \clopup(X)$ such that $C \subseteq V \ll U$. Thus,  \cref{lem:Scott-inbetween} yields
    a Scott upset $F$
    such that $C \subseteq V \subseteq F \subseteq U$.

    ($\Leftarrow$) Suppose $C \subseteq F \subseteq U$ for some $F \in \SUp(X)$. Since $X$ is a continuous \L-space, ${F \subseteq \cl \con U}$. Thus, 
    $C \subseteq F \subseteq \con U$
    by \cref{Scott-upsets-property}. 
\end{proof}

\begin{lemma} \label{lem:con-and-joins}
    Let $X$ be a continuous \L-space and $U = \cl \bigcup \mathcal U$ for some $\mathcal U \subseteq \clopup(X)$. Then $\con U = \bigcup\{\con V \mid V \in \mathcal U\}$. In particular, $\con$ commutes with finite unions of clopen upsets.
\end{lemma}

\begin{proof}
    Since $V \subseteq U$ for each $V \in \mathcal U$ and $\con$ is monotone, the $\supseteq$ inclusion is clear.
    For the other
    inclusion, suppose $x \in \con U$. 
    By \cref{lem:con-Scott}, there is $F \in \SUp(X)$ such that $x \in F \subseteq \cl \bigcup \mathcal U$. By \cref{Scott-upsets-property}, 
    $F \subseteq \bigcup \mathcal U$, and hence $x \in \bigcup \mathcal U$. By \cref{lem:min-max-nonempty}, there is $y \in \min F \subseteq \loc X$ with $y \leq x$. Since $F \subseteq \bigcup \mathcal U$, there is $V \in \mathcal U$ with $y \in V$. Therefore, $x \in \upset y \subseteq V$. Because $\upset y \in \SUp(X)$, another application of \cref{lem:con-Scott} yields that $x \in \con V$.
    
    In particular,
    if $U_1, \dots, U_n \in \clopup(X)$, then $U_1 \cup \dots \cup U_n 
    = \cl (U_1 \cup \dots \cup U_n)$, so the above gives that   
    $\con (U_1 \cup \dots \cup U_n)
    = \con  U_1 \cup \dots \cup\con U_n$.
\end{proof}

\begin{lemma}[{\cite[Lem.~6.3(1)]{BM23}}]
    \label{lem:Scott-intersection-con}
    Let $X$ be a continuous \L-space and $F \in \SUp(X)$. Then
    $
        F = \bigcap \{\con U \mid F\subseteq U \in \clopup(X)\}.
    $
\end{lemma}

We use the preceding lemmas to obtain the following characterization of $\loc X_{\d}$.

\begin{theorem} \label{thm:loc-Xd-eqv-conditions}
    Let $X$ be a stably continuous \L-space and $x \in \loc X$. The following are equivalent.
    \begin{enumerate}[cref=corollary]
        \item $x \in \loc X_{\d}$.\label{ddu-points-1}
        \item If $U \in \clopup(X)$ and $ \max \upset x \subseteq \con U$, then $x \in \con U$.\label{ddu-points-2}
        \item $x$ is the greatest localic point below some maximal point of $X$.\label{ddu-points-5}
        \item If $F \in \SUp(X)$ and $\max \upset x \subseteq F$, then $x \in F$.\label{ddu-points-4}
    \end{enumerate}
\end{theorem}

\begin{proof}
    \tref{ddu-points-1}$\Rightarrow$\tref{ddu-points-2}. Suppose $U \in \clopup(X)$ and $\max \upset x \subseteq \con U$. By \tref{ddu-points-1}, ${x \in \loc X_{\d}}$, so \cref{cor:char-of-Xd} applies, by which $x \in U$ 
    (take $y = x$ in \ref{new-Xd-3}).
    Since $U = \cl \con U$ and $x \in \loc X$, we conclude that 
    $x \in \con U$ (see \cref{rem:loc-X-and-open-upsets}).

    \tref{ddu-points-2}$\Rightarrow$\tref{ddu-points-5}. Suppose $x$ is not the greatest localic point below some maximal point of $X$. Then for each $y \in \max \upset x$, there is $z_y \in \loc X$ such that $x < z_y \leq y$.  Therefore, ${\upset z_y \in \SUp(X)}$, so \cref{lem:Scott-intersection-con} yields some $U_y \in \clopup(X)$ such that $z_y \in \con U_y$ and $x \notin \con U_y$. Thus, ${\max \upset x \subseteq \bigcup \con U_y}$ and by compactness there exist finitely many $U_{y_1},\dots,U_{y_n}$ such that $${\max \upset x \subseteq \con U_{y_1} \cup \dots \cup \con U_{y_n}}.$$ Letting $U =  U_{y_1} \cup \dots \cup U_{y_n}$ and using  \cref{lem:con-and-joins}, $\max \upset x \subseteq \con U$ and $x \notin \con U$,
    contradicting~\tref{ddu-points-2}.

    \tref{ddu-points-5}$\Rightarrow$\tref{ddu-points-4} 
    Suppose $F \in \SUp(X)$ and $\max \upset x \subseteq F$. By \tref{ddu-points-5}, there is $m \in \max X$ such that $x$ is the greatest localic point below $m$. Therefore, $m \in \max \upset x \subseteq F$. Since $F$ is a Scott upset, there is $y \in F \cap \loc X$ with $y \le m$. Because $x$ is the greatest localic point below $m$, we must have
    $y \leq x$, yielding that $x \in F$.

\tref{ddu-points-4}$\Rightarrow$\tref{ddu-points-1}
 It suffices to show that $x \in X_{\d}$ since $\loc X_{\d} = X_{\d} \cap \loc X$ (see \cref{rem:nuclear-and-priestley}). By \cref{cor:char-of-Xd}, it is enough to show that \ref{new-Xd-3} is satisfied.
 Suppose $U \in \clopup(X)$ and for each clopen downset $D$ containing $x$ there is $y \in D \cap \loc X$ such that $\max \upset y \subseteq \con U$. Since $x \in \loc X$, $\downset x$ is a clopen downset, so there is $y \in \downset x \cap \loc X$ with $\max \upset y \subseteq \con U$. 
 Therefore, $\max \upset x \subseteq \max \upset y \subseteq \con U$ and since $\max \upset x$ is closed, by \cref{lem:con-Scott} there is  $F \in \SUp(X)$ such that $\max \upset x \subseteq F \subseteq U$. Thus, $x \in F \subseteq U$ by \tref{ddu-points-4}. Consequently, \ref{new-Xd-3} is satisfied.
\end{proof}

In the above theorem, \ref{ddu-points-5} is especially noteworthy. Restricting it 
to algebraic \L-spaces yields the following characterization of \cite{BBM25}:

\begin{corollary}[{\cite[Thm.~4.12]{BBM25}}]
        Let $X$ be an arithmetic \L-space.
        Then
        \[
        \loc X_d = \{x \in \loc X \mid x \text{ is the greatest localic point below some maximal point of $X$}\}.
        \]
\end{corollary}

\begin{remark}\label{rem: min primes}
Another consequence of \cref{thm:loc-Xd-eqv-conditions} is that for each stably continuous frame~$L$ and dense Scott-continuous nucleus $j \in \N(L)$, the minimal primes of $L$ are minimal primes of $jL$, generalizing the observation of \cite[Rem.~4.5]{MZ03}. This can be seen as follows. Let $X$ be the \L-space of $L$. Then minimal primes of $L$ are in dual correspondence with $\max (\loc X)$ (see, e.g., \cite[pp.~13--14]{PP12}). By \cref{lem:min-max-nonempty}, every $x \in \max (\loc X)$ is beneath some maximal point $z$. We show that $x$ is the greatest localic point beneath $z$. Suppose $y$ is a localic point beneath $z$. Since $X$ is \L-stably continuous, $\upset x \cap \upset y$ is a Scott upset and it is nonempty since $z \in \upset x \cap \upset y$. Therefore, it contains a localic point, which must be above $x$. Because $x$ is a maximal localic point, we must have $\upset x \cap \upset y = \upset x$, so $y \le x$, and hence $x$ is the largest localic point beneath $z$. 
Hence,
$x \in \loc X_{\d}$ by \cref{thm:loc-Xd-eqv-conditions}. Since $j \leq \d$ by \cref{largest-finitary}, it follows from \cref{thm:nuclei-and-sublocales} that $X_{\d} \subseteq X_j$. Since $\loc X_j = X_j \cap \loc X$ and $\loc X_{\d} = X_{\d} \cap \loc X$ (see \cref{rem:nuclear-and-priestley}), we have $x \in \loc X_{\d} \subseteq \loc X_j$. Because $x \in \max (\loc X)$, it follows that $x \in \max (\loc X_j)$, and hence its corresponding minimal prime is a minimal prime of $jL$.
\end{remark}

We conclude this section by observing  
that the $d$-nucleus provides means to realize the Gleason cover of each compact Hausdorff space, thus giving
another perspective on this important construction. 
We recall that a subset $U$ of a topological space $X$ is \emph{regular open} if $U = \intr (\cl U)$, where $\intr$ is the topological interior. Equivalently, $U$ is regular open provided $U$ is a fixpoint of the double-negation nucleus on the frame of opens $\Omega(X)$. Thus, the regular opens of $X$ are precisely the booleanization $\B(\Omega(X))$ of its frame of opens. 

The \emph{Gleason cover} of a compact Hausdorff space $X$ is the pair $(\mathcal G X,g)$ where $\mathcal G X$ is the \L-space of $\B(\Omega(X))$ and $g : \mathcal G X \to X$ is defined as follows. Let $P \in \mathcal G X$. Since $X$ is compact Hausdorff, $\bigcap\{ \cl(U) \mid U \in P\}$ is a singleton $\{x\}$ and we define $g(P)=x$  
(see, e.g., \cite[p.~107]{Joh82}). 

\begin{theorem} \label{thm:Gleason}
    Let $X$ be compact Hausdorff, and let $L$ be the frame of ideals of $\Omega(X)$. Then the compact elements $\K(dL)$ of $dL$ form a boolean frame and $\mathcal G X$ is homeomorphic to the \L-space of $\K(dL)$.
\end{theorem}

\begin{proof}
    Since $\K(L)$ is the principal ideals of $L$, we have $\K(L) \cong \Omega(X)$.
    Because $dk=k^{**}$ for each $k \in \K(L)$, the frame of regular opens $\B(\Omega(X))$ of $X$
    is precisely $d\K(L)$.
    Therefore, $\B(\Omega(X)) \cong d\K(L) = \K(dL)$, where the equality follows from \cref{rem:KdL=dKL}.
    Thus, the Gleason cover of $X$ is the \L-space of $\K(dL)$.  
\end{proof}

\begin{discussion}\leavevmode \label{disc:gleason}
    Let $X$ be a compact Hausdorff space. As in \cref{thm:Gleason}, we let $L$ be the frame of ideals of $\Omega(X)$ and $Z$ the \L-space of $L$. 
    \begin{enumerate}[cref=discussion]
        \item \label{rem: loc Zd} As was observed in \cref{thm:Gleason}, $\K(dL)$ is a boolean frame. Since $dL$ is a coherent frame (see \cref{closure-under-Scott-cts}), $\loc Z_d$ is homeomorphic to the \L-space of $\K(dL)$ (see, e.g., \cite[Rem.~3]{Ban80} or \cite[p.~65]{Joh82}). Therefore, by \cref{thm:Gleason}, $\mathcal G X$ is homeomorphic to $\loc Z_d$.  
        \item\label{disc:gleason-3}
        The Gleason cover of 
        $X$ admits an alternative description (see \cite{BGJ16}). Let $Y$ be the \L-space of $\Omega(X)$. 
        For each $x \in Y$ there is a unique $y \in \min Y$ such that $y \leq x$ (see \cite[Lem.~5.12]{BGJ16}). Moreover, $X \cong \loc Y = \min Y$, and the Gleason cover of $X$ can be realized as the pair $(\max Y,\pi)$, where $\pi : \max Y \to \loc Y$ is defined by
        \[
            \pi(x)=\text{the unique minimal point beneath }x
        \]
        (see \cite[Thm.~5.13]{BGJ16}).
        By \tref{rem: loc Zd}, $Y \cong \loc Z$.
        Thus, 
        $
        \max Y \cong \max(\loc Z), 
        $
        and hence $\mathcal G X$ is homeomorphic to $\max(\loc Z)$.
    \end{enumerate}
We continue this discussion in the next section (see \cref{disc: Gleason continued}), where we provide yet another description of the Gleason cover of $X$ using the \emph{maximal $d$-spectrum} of $L$.
\end{discussion}

\section{The maximal \texorpdfstring{$\d$}{d}-spectrum and locally compact Hausdorff spaces} \label{sec:d-spectra}

In the theory of archimedean Riesz spaces, \cite{HdP83} 
introduced the spectrum of maximal $d$-ideals and showed that, in the presence of a weak order unit, it is a compact Hausdorff space. This was later generalized in \cite{BM18} to archimedean $\ell$-groups. In the setting of arithmetic frames, \cite{Bha19} reformulated this spectrum via the $d$-nucleus of \cite{MZ03}. 
Related spectra 
had been considered previously. For instance, \cite{HM07} studied the space of points of $dL$ (i.e., $\loc X_d$), while \cite{Dub13} investigated the spectrum of minimal $d$-elements. 
A systematic study
of the maximal $d$-spectrum of an arithmetic frame was initiated 
in \cite{Bha19}.
We recall that an element of a frame is \emph{proper} if it is not the top element, and that it is \emph{maximal} if there is no proper element above it.

\begin{definition} \label{def:max-spec} 
    Let $L$ be a frame.
    \begin{enumerate}
        \item We write $\max(L)$ for the set of maximal elements of $L$.
        \item We equip $\max(L)$ with
        the topology $\{\max(L) \setminus \upset a \mid a \in L\}$. 
        \item If $j : L \to L$ is a nucleus, then the space $\max(jL)$ is called the \emph{maximal $j$-spectrum} of~$L$.
    \end{enumerate} 
\end{definition}

The following is straightforward. 

\begin{lemma}\label{thm:max-dL-basic}
    Let $L$ be a frame.
\begin{enumerate}[cref=theorem]
    \item $\max(L)$ is a $T_1$-space.\label{prop:max-ddL-T1}
    \item If $L$ is continuous, then 
    each open of $\max(L)$ is of the form $\bigcup\{\max(L) \setminus \upset b \mid b\ll a\}$ for some $ a \in L$. 
    \label{prop:gen-basis}
\end{enumerate}
\end{lemma}
\begin{proof}
    \tref{prop:max-ddL-T1} 
    If $a,b \in \max(L)$ are distinct, then 
    $a \nleq b$, so $b \in U := \max(L) \setminus \upset a$, but $a \notin U$.
    
     \tref{prop:gen-basis}
    By \cref{upset-intersection-ll}, $\upset a = \bigcap \{\upset b \mid b\ll a\}$. 
    Thus,
    \[
        \max(L) \setminus \upset a  = \max(L) \setminus \bigcap \{\upset b \mid b \ll a\} = \bigcup\{ \max(L) \setminus \upset b \mid b \ll a\}. \qedhere
    \]
\end{proof}

\begin{remark}
    If $X$ is a $T_1$-space and $L$ is the frame of ideals of $\Omega(X)$, then $\max(L)$ is the \emph{Wallman compactification} of $X$ \cite{Wal38}. Therefore, 
    $\max(L)$ is compact and $T_1$, and each compact $T_1$-space is realized this way. 
\end{remark}

If $j$ is a nucleus on $L$, then the above yields that $\max(jL)$ is a $T_1$-space, and if $L$ is continuous, then the opens of $\max(jL)$ are of the form $\bigcup\{\max(jL) \setminus \upset b \mid b\ll a\}$ for some $ a \in L$. This applies to the nucleus $\d$ on a stably continuous frame, as well as to the nucleus $d$ on an arithmetic frame. 
For an arithmetic frame $L$, the space $\max(dL)$ has been studied in \cite{Bha19,Bha24,BBM25}. A basic problem is to determine which spaces arise as maximal $d$-spectra of arithmetic frames. 
In this section, we 
show that every locally compact Hausdorff space can be realized as a maximal $\d$-spectrum of a stably continuous frame. 
As a consequence, we obtain that every locally Stone space
arises as the maximal $d$-spectrum of an arithmetic frame. 

We start by 
interpreting the maximal $\d$-spectrum using the language of Priestley duality. 
For this we require the following lemma.

\begin{lemma} \label{lem:min-loc-X-iso}
    If $X$ is a spatial \L-space, then
    every minimal localic point is a minimal point of~$X$.
\end{lemma}

\begin{proof}
    Let $x \in \min (\loc X)$ and $y \leq x$. If $x \nleq y$, then there is a clopen downset $D$ such that $y \in D$ and $x \notin D$. Since $X$ is a spatial \L-space and $D \cap \downset x \neq \varnothing$, there is $z \in D \cap \downset x \cap \loc X$. But then $z = x$ since $x \in \min (\loc X)$, contradicting that $x\notin D$. Therefore, $y = x$, and hence $x \in \min X$.
\end{proof}

\begin{proposition} \label{prop:max-L-and-min-loc-X}
    If $X$ is a spatial \L-space, then the map $x \mapsto X \setminus \{x\}$ is a bijection from $\min (\loc X)$ to $\max (\clopup(X))$. In particular, if $X$ is the \L-space of a spatial frame $L$, then $\max L$ is in bijective correspondence with $\min (\loc X)$.
\end{proposition}

\begin{proof}
    The map $x \mapsto X \setminus \{x\}$ is well defined by \cref{lem:min-loc-X-iso}, and it is clearly injective. To see that it is surjective, let $U \in \max (\clopup(X))$. Then $U \neq X$, and since $X$ is \L-spatial, there is $x \in \loc X$ such that $x\notin U$. Therefore, $\downset x \cap U = \varnothing$, 
    and so $U \subseteq X \setminus \downset x$. Thus, $U = X \setminus \downset x$ by the maximality of $U$.
    Moreover, if $y \in \loc X$ with $y \le x$, then $U \subseteq X \setminus \downset y$, so $U = X \setminus \downset y$ by maximality, and hence $y = x$. Consequently, $x \in \min (\loc X)$, so $x \in \min X$ by \cref{lem:min-loc-X-iso}, and thus $U = X \setminus \{x\}$.
\end{proof}

Viewing $\min(\loc X)$ as a subspace of $\loc X$, we have:

\begin{theorem} \label{thm:min-loc-X-homeomorphic-max-L}
    Let $L$ be a spatial frame and $X$ its \L-space. Then $\min(\loc X)$ is homeomorphic to $\max(L)$.
\end{theorem}
\begin{proof}
    By identifying $L$ with $\clopup(X)$ and applying \cref{prop:max-L-and-min-loc-X}, it suffices to show that the map $\psi : \min (\loc X) \to \max (\clopup(X))$, defined by $\psi(x) = X \setminus \{x\}$, is continuous and open. 
    To see that $\psi$ is continuous, let $\mathcal U$ be open in $\max(\clopup(X))$. Then there is $U \in \clopup(X)$ such that 
\[
    \mathcal U = \max(\clopup(X)) \setminus \upset U 
    = \{V \in \max(\clopup(X)) \mid U \not\subseteq V\}.
\] 
Therefore,  
\begin{align*}
    \psi^{-1}(\mathcal U) 
    &= \{ x \in \min (\loc X) \mid U \not\subseteq X \setminus \{x\}  \} \\
    &= \{ x \in \min (\loc X) \mid x \in U\} \\
    &= U \cap \min (\loc X),
\end{align*}
which is open in $\min (\loc X)$.

To see that $\psi$ is open, let $U \in \clopup(X)$, so that $U \cap \min (\loc X)$ is open in $\min(\loc X)$. Then
\begin{align*}
    \psi[U \cap \min (\loc X)] 
    &= \{X \setminus \{x\} \mid x \in U \cap \min (\loc X)\} \\
    &= \{V \in \max(\clopup(X)) \mid U \not\subseteq V\} \\
    &= \max(\clopup(X)) \setminus \upset U,
\end{align*}
which is open in $\max(\clopup(X))$.
\end{proof}

\begin{corollary} \label{cor:min-loc-Xd=max-dL}
    Let $L$ be a stably continuous frame and $X$ its \L-space. Then $\min(\loc X_{\d})$ is homeomorphic to $\max(\d L)$.
\end{corollary}
\begin{proof}
    Observe that $\d L$ is stably continuous by \cref{closure-under-Scott-cts}. Thus, $\d L$ is spatial and \cref{thm:min-loc-X-homeomorphic-max-L} applies.
\end{proof}
In particular, if $L$ is an arithmetic frame, then we obtain the characterization of \cite[Thm.~6.6]{BBM25} that $\min(\loc X_d)$ is homeomorphic to $\max(d L)$. 

\needspace{2em}
\begin{discussion}\label{disc: Gleason continued}
Let $X$ be a compact Hausdorff space.
\begin{enumerate}
\item 
Let $L$ be the frame of ideals of $\Omega(X)$ and let $Z$ be the \L-space of $L$. As we pointed out in \cref{rem: loc Zd}, $\loc Z_d$ is the Gleason cover of $X$. 
Since $\loc Z_d$ is the \L-space of $\K(dL)$, which
is a boolean frame by \cref{thm:Gleason},
we have that $\min (\loc Z_d) = \loc Z_d$. Therefore, $\max(dL)$ is the Gleason cover of $X$ by \cref{cor:min-loc-Xd=max-dL,c-3}. 
But $dL$ consists exactly of the $d$-ideals of $\Omega(X)$,
so the space of maximal $d$-ideals of $\Omega(X)$ is homeomorphic to the Gleason cover of~$X$. 

\item Putting the above together with \cref{disc:gleason},
we obtain the following equivalent descriptions of the Gleason cover of 
$X$: 
\begin{enumerate}[1.]
    \item the Stone space of the regular opens of $X$ \cite{Gle58};
    \item the space of maximal points of the \L-space of $\Omega(X)$ \cite[Thm.~5.13]{BGJ16};
    \item the \L-space of the compact $d$-ideals of $\Omega(X)$ (\cref{thm:Gleason});
    \item the space of maximal $d$-ideals of $\Omega(X)$.
\end{enumerate}
\end{enumerate}
\end{discussion}

Before we investigate the topological properties of $\min (\loc X_{\d})$, we point out that spatiality can not be dropped from \cref{thm:min-loc-X-homeomorphic-max-L}.

\begin{example}
Let $B$ be a complete atomless boolean algebra, and let $L$ be the frame obtained by adjoining a new bottom to $B$; see \cref{fig:spatiality-b}. Then $\max (L) = \varnothing$. On the other hand, the \L-space $X$ of $L$ is obtained by adjoining a new top to the \L-space $X_B$ of $B$. The latter is a dense-in-itself antichain since $B$ is an atomless boolean algebra; see \cref{fig:spatiality-a}. Therefore, no $x \in X_B$ is a localic point, and hence $\loc X = \{x\} = \min (\loc X)$.
This shows that spatiality is required in \cref{thm:min-loc-X-homeomorphic-max-L} (as well as in \cref{lem:min-loc-X-iso,prop:max-L-and-min-loc-X}).
        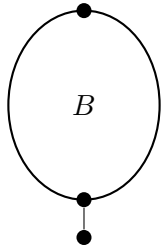
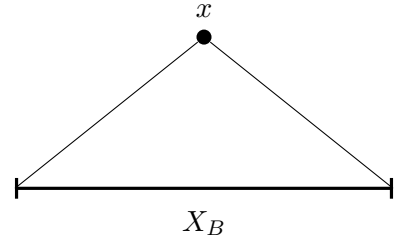
\begin{figure}[H]
    \begin{subfigure}{.48\textwidth}
    \centering
    \begin{tikzpicture}[
    bullet/.style={fill,circle,inner sep=2pt}]
        \draw[thick] (0,0) ellipse (1 and 1.25) node at (0,0) {$B$};
        \begin{scope}[nodes=bullet]
        \node (0) at (0,-1.25) {};
        \node (1) at (0,1.25) {};
        \node (2) at (0,-1.75) {};
        \draw (0) to (2);
        \end{scope}
    \end{tikzpicture}
    \subcaption{The frame $L$ obtained from $B$ by adjoining a new bottom}
\label{fig:spatiality-b}
    \end{subfigure}\hfill
    \begin{subfigure}{.48\textwidth}
    \centering
    \begin{tikzpicture}[
    bullet/.style={fill,circle,inner sep=2pt}]
    \begin{scope}[nodes=bullet]
        \node[label=above:$x$] (y) at (3.5,2) {};
    \end{scope}
        \draw[very thick, |-|] (1, 0) to node[label=below:$X_B$] {} (6,0);
        \draw (y) to (1,0);
        \draw (y) to (6,0);
    \end{tikzpicture}
    \subcaption{The \L-space $X$ obtained from the Stone space $X_B$ by adjoining a new top $x$}
    \label{fig:spatiality-a}
    \end{subfigure}
    \caption{A non-spatial frame $L$ such that $\max (L)$ is not homeomorphic to $\min (\loc X)$} 
    \end{figure}
\end{example}

We now study some topological properties of $\min (\loc X_{\d})$, which we identify with $\max(\d L)$. Since $\max(\d L) = \max(d L)$ for an arithmetic frame $L$, in general $\max(\d L)$ is neither Hausdorff (see  \cite[Example~8.1]{BBM25}) nor compact (see \cite[Example~3.7]{Bha24}).  In fact, $\max(\d L)$ might even be empty (see \cite[Example~8.11]{BBM25}). For this reason, it is customary to assume that an arithmetic frame $L$ has a \emph{unit} since then $\max(d L)$ is a nonempty compact space \cite[Thm.~3.7]{Bha19}. 
We recall:  

\begin{definition}
An element $a$ of a frame $L$ is \emph{dense} if $a^*=0$, and it is a \emph{unit} if it is compact and dense. 
\end{definition}

\begin{example}
In a compact frame, the top element is a unit. 
For an example of a 
non-compact frame 
with a unit, 
let $L = \omega + 1$. It is straightforward to see that $L$ is not compact, but any nonzero $n \in \omega$ is a unit of $L$.  
\end{example}

For an arithmetic frame $L$, the existence of a unit is in general stronger than $\max(\d L)$ being a compact space. 
We recall (see, e.g., \cite[p.~25]{PP21}) that a frame is \emph{max-bounded} if each proper element lies below a maximal element. 
By \cite[Thm.~7.7]{BBM25}, if $L$ is arithmetic and $\d L$ is max-bounded then $\max(\d L)$ is compact iff $L$ has a unit. Since $L$ having a unit implies that $\d L$ is max-bounded (see, e.g., \cite[before Prop.~3.3]{Bha19}), it follows that for an arithmetic frame $L$, $L$ has a unit iff $\max(\d L)$ is compact and $\d L$ is max-bounded. This equivalence fails for stably continuous frames, indicating that the notion of a unit is too restrictive in this setting. We therefore introduce a suitable generalization.

\begin{definition}
    We call an element $a$ of a frame $L$ a \emph{Scott unit}, or an \emph{S-unit} for short, provided $\twoheaduparrow a := \{b \in L \mid a \ll b\}$ is a Scott-open filter consisting only of dense elements.
\end{definition}

\begin{example} \label{example-S-unit}
    Every unit is an S-unit. Indeed, if $a$ is a unit then $\twoheaduparrow a = \upset a$ is a Scott-open filter since $a$ is compact, and it consists only of dense elements since $a$ is dense. However, there are S-units that are not units. Let $L = [0,1]$ be the unit interval with the usual order. Then for each $a \in L$, 
    \[
        a = \bigvee\{b \in L \mid b < a\}.
    \]
    Consequently, $\K(L) = \{0\}$, and hence there are no units in $L$. 
    Moreover, if $a \neq 0$, then $a$ is dense. Therefore, every $a \in (0,1)$ is an $S$-unit since $\twoheaduparrow a =(a,1]$ is a Scott-open filter.
\end{example}

Observe that the frame in the above example is stably continuous. Thus, there exist stably continuous frames with many S-units that have no units at all. In contrast, we show that in  arithmetic frames, having a unit is equivalent to the existence of an S-unit. 

\begin{proposition}\label{prop:units-iff-S-unit}
    Let $L$ be a stably continuous frame. The condition
    \begin{enumerate}[cref=proposition]
        \item $L$ has a unit \label{prop:unit}
    \end{enumerate}
    implies the following equivalent conditions.
    \begin{enumerate}[cref=proposition, resume]
        \item $L$ has an S-unit.
        \label{prop:S-unit-1}
        \item There exists a Scott-open filter $\filter{F}$ of $L$ consisting only of dense elements.\label{prop:S-unit-2}
    \end{enumerate}
    If $L$ is algebraic, 
    all three conditions are equivalent.
\end{proposition}

\begin{proof}
    \tref{prop:unit}$\Rightarrow$\tref{prop:S-unit-1} Every unit is an S-unit (see \cref{example-S-unit}).

    \tref{prop:S-unit-1}$\Rightarrow$\tref{prop:S-unit-2} By definition, 
    $\twoheaduparrow a$ is a Scott-open filter consisting only of dense elements. 

    \tref{prop:S-unit-2}$\Rightarrow$\tref{prop:S-unit-1} 
    Since $\filter{F}$ is a Scott-open filter and $1 \in \filter{F}$, 
    \cref{lem:round-filters} implies that there is $a \in \filter{F}$ with $a \ll 1$.
    Therefore, $\twoheaduparrow a$ is a Scott-open filter by \cref{lem:upup-Scott}. Moreover, since $a \in \filter{F}$, $\twoheaduparrow a \subseteq \filter{F}$, so $\twoheaduparrow a$ consists only of dense elements.

    \tref{prop:S-unit-1}$\Rightarrow$\tref{prop:unit} Let $L$ be algebraic. 
    Since $\twoheaduparrow a \neq \varnothing$, there is $b \in \twoheaduparrow a$. Because $L$ is algebraic, ${b = \bigvee\{k \in \K(L) \mid k \leq b\}}$, which is directed, so from $a \ll b$ it follows that there is $k \in \K(L)$ with 
    $a \leq k$.
    Thus, 
    $k$ is dense, and hence $k$ is a unit.
\end{proof}

\begin{lemma}[Basic properties of S-units]
    Let $L$ be a stably continuous frame with an S-unit~$e$.
    \begin{enumerate}[cref=lemma]
        \item $e \ll a$ implies $\d a = 1$. \label{lem:S-unit-ll-implies-1}
        \item $e \ll \d a$ implies $f \ll a$ for some S-unit $f$. \label{lem:S-unit-below}
    \end{enumerate}
\end{lemma}

\begin{proof}
\tref{lem:S-unit-ll-implies-1} Let $e \ll a$. By \cref{lem:interpolating}, $e \ll b \ll a$ for some $b \in L$. Therefore, $b \in \twoheaduparrow e$, so $1 = b^{**} \leq \d a$.

\tref{lem:S-unit-below} Let $e \ll \d a$. By \cref{lem:interpolating}, there exists $b \in L$ such that $e \ll b \ll \d a$. Thus, $b \ll \bigvee\{c^{**} \mid c \ll a\}$ and the join is directed. 
Therefore, there is $f \ll a$ such that $b \leq f^{**}$. 
By \cref{lem:upup-Scott}, $\twoheaduparrow f$ is a Scott-open filter, and $f$ is dense since $f^{**} \in \twoheaduparrow e$ by \cref{lem:weakening-2}. Thus,
$f$ is an S-unit. 
\end{proof}

\begin{proposition} \label{prop:S-unit-implies-max-bounded-and-compact}
    Let $L$ be a stably continuous frame with an S-unit $e$. 
    \begin{enumerate}[cref=proposition]
        \item $\d L$ is max-bounded. \label{prop:S-unit-implies-dL-max-bounded}
        \item $\max(\d L)$ is compact. \label{prop:S-unit-implies-max-dL-compact}
    \end{enumerate}
\end{proposition}

\begin{proof}
    \tref{prop:S-unit-implies-dL-max-bounded} Let $a \in \d L \setminus \{1\}$, and let $Z = (\d L \setminus \{1\}) \cap \upset a$. It suffices to show that $\max Z \neq \varnothing$. Suppose $C \subseteq Z$ is a chain. Let $c = \bigvee C$. If $d \ll c$, then $d \leq b$ for some $b \in C$. 
    Thus, $d^{**} \leq \d b =b $, so $\d c \leq c$, and hence $c \in \d L$. If $c = 1$, then $\bigvee C \in \twoheaduparrow e$, so there is $b \in C$ with $b \in \twoheaduparrow e$. By \cref{lem:interpolating}, there is $t$ with $e \ll t \ll b \leq c$. Since $e$ is an S-unit, $t$ is dense, so $t^{**} = 1 \leq \d b = b$, contradicting that $b \in Z$. Applying Zorn's Lemma completes the proof. 

    \tref{prop:S-unit-implies-max-dL-compact} Suppose $\max(\d L) = \bigcup \{\max(\d L) \setminus \upset s \mid s \in S\}$ for some $S \subseteq \d L$. If $\bigvee_{\d L} S \neq 1$, then since $\d L$ is max-bounded, there is $a \in \max (\d L)$ with $\bigvee_{\d L} S \leq a$. This implies that $s \leq a$ for each $s \in S$, so $a \notin \max(\d L) \setminus \upset s$ for each $s \in S$, a contradiction. Therefore, $\bigvee_{\d L} S = \d (\bigvee S) \in \twoheaduparrow e$. By \cref{lem:S-unit-below}, there is an S-unit $f \ll \bigvee S$. 
    Since $\twoheaduparrow f$ is a Scott-open filter, \cref{lem:SFilt-finite-joins} yields that $f \ll \bigvee T$ for some finite $T \subseteq S$.
    Thus, by \cref{lem:S-unit-ll-implies-1}, $\d (\bigvee T) = \bigvee_{\d L} T = 1$. 
    We claim that $\bigcup \{\max(\d L) \setminus \upset t \mid t \in T\} = \max(\d L)$. Let $b \in \max(\d L)$. 
    Since $\bigvee_{\d L} T =1 \nleq b$, there is $t \in T$ with $t \nleq b$. Thus, $b \in \max(\d L) \setminus \upset t$, as required.
\end{proof}

In \cref{thm:max-dL-compactness}, we will give a dual description of an S-unit, which yields the converse of \cref{prop:S-unit-implies-max-bounded-and-compact}. This requires some preparation.
Recall that if $X$ is the \L-space of a frame $L$, 
then for each $F \in \clup(X)$, the set $\Phi(F) \in \Filt(L)$ is the corresponding filter (see \cref{rem:closed-upsets-and-filters}). 
It is straightforward to see that $a \in L$ is dense iff $\max X \subseteq \sigma(a)$ (see, e.g., \cite[Rem.~4.4]{BGJ16}). Therefore, 
$\Phi(F)$ consists only of dense elements iff $\max X \subseteq F$.

\begin{lemma} \label{lem:dense-filter}
    Let $L$ be a stably continuous frame, $X$ its \L-space, and $F \in \SUp(X)$. The following are equivalent.
    \begin{enumerate}[cref=lemma]
        \item $\loc X_{\d} \subseteq F$. \label{lem:dense-filter-1}
        \item $X_{\d} \subseteq F$. \label{lem:dense-filter-1.5}
        \item $\max X \subseteq F$.\label{lem:dense-filter-2}
        \item $\Phi(F)$ consists only of dense elements.\label{lem:dense-filter-3}
    \end{enumerate}
\end{lemma}

\begin{proof}
    \tref{lem:dense-filter-1}$\Rightarrow$\tref{lem:dense-filter-1.5} 
    Let $\loc X_{\d} \subseteq F$. Since $\d L$ is stably continuous, it is spatial, so ${X_{\d} = \cl (\loc X_{\d})}$ by \cref{thm:spatiality}. Thus, since $F$ is closed, $X_{\d} \subseteq F$.
    
    \tref{lem:dense-filter-1.5}$\Rightarrow$\tref{lem:dense-filter-2} 
    Suppose $X_{\d} \subseteq F$. 
    Since $\d$ is dense,  $\max X \subseteq X_{\d}$ by \cref{thm:dense}. Therefore, $\max X \subseteq F$.

    \tref{lem:dense-filter-2}$\Rightarrow$\tref{lem:dense-filter-1} 
    Let $x \in \loc X_{\d}$. 
    Then $\max \upset x \subseteq \max X \subseteq F$, so $x\in F$ by \cref{ddu-points-4}. 

    \tref{lem:dense-filter-2}$\Leftrightarrow$\tref{lem:dense-filter-3} This is explained in the paragraph before the lemma.
\end{proof}

\begin{proposition} \label{lem:compactness}
    For a stably continuous \L-space $X$, $\min (\loc X_{\d})$ is compact iff $\upset \min (\loc X_{\d})$ is a Scott upset.
\end{proposition}

\begin{proof}
    ($\Rightarrow$) 
    Clearly, $F:= \upset \min (\loc X_{\d})$ is an upset. Also, by \cref{rem:nuclear-and-priestley}, ${\loc X_{\d} = X_{\d} \cap \loc X}$, and hence  $\min F = \min (\loc X_{\d}) \subseteq \loc X$. It remains to show that $F$ is closed. If $x \notin F$, then $y \nleq x$ for all $y \in \min (\loc X_{\d})$, so by Priestley separation, there is $U_y \in \clopup(X)$ such that $y \in U_y$ and $x \notin U_y$. Therefore, $\min(\loc X_{\d}) \subseteq \bigcup U_y$, so $\min(\loc X_{\d}) \subseteq \bigcup (U_y \cap \min (\loc X_{\d}))$. Since $\min(\loc X_{\d})$ is compact, there is $U \in \clopup(X)$ with $\min(\loc X_{\d}) \subseteq U$ and $x\notin U$. Because $U$ is an upset, $F \subseteq U$, so $F$ is closed, and hence $F$ is a Scott upset.

    ($\Leftarrow$) 
    Suppose $\min(\loc X_{\d}) \subseteq \bigcup \{U \cap \min (\loc X_{\d}) \mid U \in \mathcal U\}$ for some $\mathcal U \subseteq \clopup(X)$. 
    Then $\upset \min(\loc X_{\d}) \subseteq \bigcup \mathcal U$, and since $\upset \min (\loc X_{\d})$ is closed, by compactness there is a finite $\mathcal V \subseteq \mathcal U$ such that $\min (\loc X_{\d}) \subseteq \bigcup \mathcal V$. Thus, $\min (\loc X_{\d}) \subseteq \bigcup\{U \cap \min(\loc X_{\d}) \mid U \in \mathcal V\}$, and hence $\min (\loc X_{\d})$ is compact.
\end{proof}

\begin{proposition} \label{lem:max-bounded}
    Let $L$ be a spatial frame and $X$ its \L-space. Then $L$ is max-bounded iff $\loc X \subseteq \upset \min (\loc X)$.
\end{proposition}
\begin{proof}
    It suffices to show that $\clopup(X)$ is max-bounded iff $\loc X \subseteq \upset \min (\loc X)$.

    ($\Rightarrow$) 
    Let $x \in \loc X$. Then $U:= X \setminus \downset x \in \clopup(X)$. Since $U$ is proper and $\clopup(X)$ is max-bounded, there is $V \in \max (\clopup(X))$ with $U \subseteq V$. 
    By \cref{prop:max-L-and-min-loc-X}, there is $y \in \min (\loc X)$ with $V = X \setminus \{y\}$. 
    Thus, $y \in \downset x$, so $y \leq x$, and hence $\loc X \subseteq \upset \min (\loc X)$.

    ($\Leftarrow$) Let $U \in \clopup(X)$ be proper. Since $X$ is \L-spatial, there is $x \in \loc X \setminus U$. By assumption, there is $y \in \min (\loc X)$ with $y \leq x$. By \cref{prop:max-L-and-min-loc-X}, $V := X \setminus \{y\} \in \max (\clopup(X))$. Moreover, $U \subseteq V$ since $y \notin U$, and hence $\clopup(X)$ is max-bounded.
\end{proof}

\begin{theorem}\label{thm:max-dL-compactness}
    Let $L$ be a stably continuous frame and $X$ its \L-space. The following are equivalent.
    \begin{enumerate}[cref=theorem]
        \item $L$ has an S-unit. \label{thm:S-unit-iffs-1}
        \item There exists $F \in \SUp(X)$ such that $\max X \subseteq F$.\label{thm:S-unit-iffs-2}
        \item $\min (\loc X_{\d})$ is compact and $\loc X_{\d} \subseteq \upset \min (\loc X_{\d})$.\label{thm:S-unit-iffs-3}
        \item $\max(\d L)$ is compact and $\d L$ is max-bounded.\label{thm:S-unit-iffs-4}
    \end{enumerate}
\end{theorem}

\begin{proof}
    \tref{thm:S-unit-iffs-1}$\Leftrightarrow$\tref{thm:S-unit-iffs-2}
    Apply \cref{prop:units-iff-S-unit,lem:dense-filter}.     

    \tref{thm:S-unit-iffs-2}$\Rightarrow$\tref{thm:S-unit-iffs-3}
    By \cref{prop:nuclei-continuous}, $F' := \upset(F \cap X_{\d})$ is a Scott upset. By \cref{thm:dense}, $\max X \subseteq X_{\d}$, and so $\max X \subseteq F'$. Therefore, $\loc X_{\d} \subseteq F'$ by \cref{lem:dense-filter}. 
    By \cref{lem:min-max-nonempty}, 
    \begin{equation}
        \loc X_{\d} \subseteq F' = \upset \min F'. \label{eq:A}
    \end{equation}
    Thus, it suffices to show that $\min F' = \min (\loc X_{\d})$. Indeed, then $\upset \min (\loc X_{\d}) = F'$ is a Scott upset, so $\min (\loc X_{\d})$ is compact by \cref{lem:compactness}, and $\loc X_{\d} \subseteq \upset \min (\loc X_{\d})$ by \eqref{eq:A}. 
    For the left-to-right inclusion,
    suppose $x \in \min F'$. Then $x \in \loc X$ since $F'$ is a Scott upset and $x \in X_{\d}$ since $\min F' \subseteq F \cap X_{\d}$. Thus, $x \in \loc X_{\d}$ by \cref{rem:nuclear-and-priestley}. Now suppose $y \in \loc X_{\d}$ and $y \leq x$. Then $y \in F'$ by \eqref{eq:A}, so $x = y$ since $x \in \min F'$. Therefore, $x \in \min (\loc X_{\d})$. 
    For the reverse inclusion,
    suppose $x \in \min (\loc X_{\d})$. Then $x \in F'$ by \eqref{eq:A}. By \cref{lem:min-max-nonempty}, there is $y \in \min F'$ with $y \leq x$. 
    Since $y \in \loc X_{\d}$, we must have $x = y$ because $x \in \min (\loc X_{\d})$. Thus, $x \in \min F'$.

    \tref{thm:S-unit-iffs-3}$\Rightarrow$\tref{thm:S-unit-iffs-2} By \cref{lem:compactness}, $F := \upset \min (\loc X_{\d})$ is a Scott upset, 
    and ${\max X \subseteq F}$ by \cref{lem:dense-filter}.
    
    \tref{thm:S-unit-iffs-3}$\Leftrightarrow$\tref{thm:S-unit-iffs-4} Apply \cref{cor:min-loc-Xd=max-dL,lem:max-bounded}.
\end{proof}

As an immediate consequence of \cref{prop:units-iff-S-unit,thm:max-dL-compactness}, we obtain:

\begin{corollary}
    If $L$ is an arithmetic frame, then $L$ has a unit iff $\max(d L)$ is compact and $d L$ is max-bounded.
\end{corollary}

We next turn our attention to the question of Hausdorffness of $\max(dL)$.
In \cite[Cor.~8.14]{BBM25}, it was shown that if $L$ is an arithmetic frame with a unit, then $\max(d L)$ is Hausdorff iff it is stably locally compact. The key ingredient was
the following lemma:

\begin{lemma}[{\cite[Lem.~8.12]{BBM25}}]
    If $X$ is a stably locally compact space, then $X$ is Hausdorff iff $X$ is $T_1$.
\end{lemma}

As an immediate consequence of the above lemma, a $T_1$-space is locally compact Hausdorff iff it is stably locally compact. 
Therefore, since $\max(\d L)$ is $T_1$, and is compact whenever $L$ has an $S$-unit, the characterization of \cite{BBM25} extends to the stably continuous setting:

\begin{theorem}
    Let $L$ be a stably continuous frame. 
    Then $\max(\d L)$ is locally compact Hausdorff iff $\max(\d L)$ is stably locally compact. 
    In particular, if $L$ has an $S$-unit, then $\max(\d L)$ is Hausdorff iff $\max(\d L)$ is stably locally compact.
\end{theorem}

\begin{remark}
    It is an open question whether $\max(\d L)$ is stably locally compact iff it is sober. By the above theorem, this would imply that $\max(\d L)$ is Hausdorff iff it is sober.
    This is related to the
    nucleus $\rho : L \to L$ corresponding to the soberification of $\max(\d L)$ (see \cite[Rem.~6.10]{BBM25}).\footnote{
        If $\d L$ is coherent, 
        the nucleus $\rho$ can be described using the nucleus $j$ introduced in 
        \cite[Sec.~1]{Joh84}, which is used to characterize the
        soberification
        of $\max(L)$ for a coherent frame $L$. 
        Applying this to $\d L$ yields a nucleus $j \in \N(\d L)$ such that $\rho = j \circ \d$.
    }
    If $\rho$ were Scott continuous, then the frame $\rho L$ would be stably continuous by \cref{closure-under-Scott-cts}, and hence $\pt(\rho L)$ would be stably locally compact by \cref{hofmann-lawson-2}. If, in addition, $\max(\d L)$ were sober, then $\max(\d L) \cong \pt(\rho L)$, and hence $\max(\d L)$ would be stably locally compact.
    Thus, showing that $\rho$ is Scott continuous would yield a positive answer to the above question.
\end{remark}

It remains open whether $\max(\d L)$ is always locally compact. However, we will show that every locally compact Hausdorff space can be realized as a maximal $\d$-spectrum. For this, we observe that the $\d$-nucleus is the identity on regular frames.

\begin{proposition} 
    Let $L$ be a continuous regular frame and $X$ its \L-space.
    \begin{enumerate}[cref=proposition]
        \item The $\d$-nucleus is the identity. \label{regular-da=a}
        \item $X = X_{\d}$. \label{X=Xd}
        \item $L$ has an S-unit iff $L$ is compact. \label{regular-S-unit-is-compact}
    \end{enumerate}
\end{proposition}
\begin{proof}

    \tref{regular-da=a} It suffices to show that $b \ll a$ implies $b^{**} \leq a$. Since $L$ is regular, $b \ll a$ implies $b \prec a$, so $b^* \vee a = 1$.
    Thus, 
    \[
    b^{**} = b^{**} \wedge (b^* \vee a) = (b^{**} \wedge b^*) \vee (b^{**} \wedge a) = b^{**} \wedge a \le a.
    \]

    \tref{X=Xd} This follows from \tref{regular-da=a} and \cref{thm:nuclei-and-sublocales} since $X$ and $X_{\d}$ are the nuclear subsets corresponding to the identity and $\d$-nucleus, respectively.

    \tref{regular-S-unit-is-compact} The right-to-left implication is immediate since $1$ being compact makes it a unit, and hence an S-unit. For the other implication, suppose $e$ is an S-unit. Then $e \ll 1$, so \cref{lem:interpolating} implies that there is $a \in L$ with $e \ll a \ll 1$. Therefore, $\d a = 1$ by \cref{lem:S-unit-ll-implies-1}. Thus, $a = 1$ by \tref{regular-da=a}, so $1 \ll 1$, and hence $L$ is compact.
\end{proof}

\begin{lemma}[{\cite[Lem.~7.15(3)]{BM23}}] \label{lem:loc-in-min}
    If $X$ is a regular \L-space, then $\loc X \subseteq \min X$.
\end{lemma}

\begin{proposition} \label{prop:equal-spectra}
    If $X$ is a continuous regular \L-space, then
    \[
        \loc X_{\d} = \loc X = \min (\loc X) = \min (\loc X_{\d}).
    \]
\end{proposition}
\begin{proof}
    By \cref{lem:loc-in-min}, 
    $\min (\loc X) = \loc X$. Now apply \cref{X=Xd} to conclude the proof.
\end{proof}

Since the localic part of the $\L$-space of a frame coincides with its space of points, combining \cref{cor:min-loc-Xd=max-dL,prop:equal-spectra} yields 
the following, where part \tref{thm:max-dl-points-2} follows from $\d = d$ (see \cref{c-3}):

\begin{theorem} \leavevmode
    Let $L$ be a regular frame.
    \begin{enumerate}[cref=theorem]
        \item If $L$ is continuous, then $\max(\d L)$ is the space of points of $L$. 
        \label{thm:max-dL-points}
        \item If $L$ is algebraic, then $\max(d L)$ is the space of points of $L$.
        \label{thm:max-dl-points-2}
    \end{enumerate}
\end{theorem}

We are finally in a position to show that every locally compact Hausdorff space arises as the maximal $\d$-spectrum.

\begin{corollary} \label{cor:main-theorme}
    Every locally compact Hausdorff space is homeomorphic to $\max (\d L)$ for some continuous regular frame $L$.
\end{corollary}
\begin{proof}
    Let $X$ be a locally compact Hausdorff space. By \cref{hofmann-lawson-LC}, $L := \Omega(X)$ is a continuous regular frame such that $X \cong \pt L$. Therefore, $\max(\d L) \cong X$ by \cref{thm:max-dL-points}. 
\end{proof}

By adapting the above argument, we obtain analogous results for compact Hausdorff spaces, which arise as maximal $\d$-spectra of compact regular frames. Similarly, in the arithmetic setting, locally Stone and Stone spaces arise from locally Stone and Stone frames, respectively.

\begin{corollary}
\leavevmode
    \begin{enumerate}
        \item Every compact Hausdorff space is homeomorphic to $\max (\d L)$ for some compact regular frame $L$.
        \item Every locally Stone space is homeomorphic to $\max (d L)$ for some locally Stone frame $L$.
        \item Every Stone space is homeomorphic to $\max (dL)$ for some Stone frame $L$.
    \end{enumerate}
\end{corollary}
\begin{proof}
    The result follows by the same argument as in \cref{cor:main-theorme}, but uses Theorems~\ref{isbell-duality} and 
    \ref{thm:Alg-dualities}(\hyperref[thm:LStone]{4},\hyperref[thm:Stone]{5}) in place of \cref{hofmann-lawson-LC}.
\end{proof}

In conclusion, we briefly summarize our main findings and remaining open problems.
We generalized the $d$-nucleus from arithmetic frames to stably continuous frames and showed that the resulting nucleus $\d$ preserves the main structural features of the classical $d$-nucleus. In particular, $\d$ is the largest Scott-continuous nucleus beneath the double-negation nucleus. We used Priestley duality to provide a convenient characterization of $\d$, which allowed us to prove that every locally compact Hausdorff space is homeomorphic to $\max(\d L)$ for some continuous regular frame $L$. As a consequence, 
we obtained that every locally Stone space
is homeomorphic to $\max(dL)$ for some arithmetic frame $L$.

It remains open whether $\max(\d L)$ or $\max(d L)$ is always stably locally compact. More generally, it remains open which spaces arise as $\max(\d L)$ for a stably continuous frame $L$ or as $\max(dL)$ for an arithmetic frame $L$. In particular, while we showed that each locally compact Hausdorff space is realized as $\max(\d L)$ for some continuous regular frame $L$, it remains open whether each locally compact Hausdorff space is realized as $\max(d L)$ for some arithmetic regular frame $L$.

\addtocontents{toc}{\SkipTocEntry}
\section*{Acknowledgments}

We are grateful to the participants of the Chapman Frame Theory seminar, Ana Belen Avilez, Fred Dashiell, 
Tega Ighedo, Peter Jipsen, and Joanne  Walters-Wayland, for their input during the visit of the first three authors.
\newcommand{\etalchar}[1]{$^{#1}$}

\end{document}

%% file: preamble.tex
\usepackage[a4paper, margin=1in]{geometry}
\usepackage[utf8]{inputenc}
\usepackage[T1]{fontenc}
\usepackage{lmodern}

\usepackage{subcaption}

\usepackage{amssymb}
\usepackage{float}

\usepackage[dvipsnames]{xcolor}
\usepackage[shortlabels]{enumitem}
\usepackage{needspace}
\usepackage{booktabs}
\usepackage{tikz-cd}
\usepackage{xspace}
\usetikzlibrary{decorations.pathmorphing}
\usetikzlibrary{decorations.markings}
\usetikzlibrary{decorations.pathreplacing}
\usetikzlibrary{arrows.meta}
\usetikzlibrary{calc,positioning}
\usetikzlibrary{arrows.meta,calc,decorations.pathmorphing,shadows.blur,positioning}

\makeatletter
\DeclareRobustCommand{\equationsymbols}[1]{%
  \expandafter\@equationsymbols\csname c@#1\endcsname
}
\def\@equationsymbols#1{%
  \ifcase#1%
  \or \ensuremath{\dagger}%
  \or \ensuremath{\ddagger}%
  \or\else\@ctrerr\fi}%
\makeatother

\usepackage{hyperref}
\hypersetup{linktoc=all,colorlinks=true,allcolors=Blue,breaklinks=true}
\usepackage[capitalize,noabbrev]{cleveref}

\DeclareRobustCommand{\SkipTocEntry}[5]{} 

\newtheorem{lemma}{Lemma}[section]
\newtheorem{theorem}[lemma]{Theorem}
\newtheorem{proposition}[lemma]{Proposition}
\newtheorem{corollary}[lemma]{Corollary}
\theoremstyle{definition}
\newtheorem{definition}[lemma]{Definition}
\newtheorem{remark}[lemma]{Remark}
\newtheorem{discussion}[lemma]{Discussion}
\newtheorem{example}[lemma]{Example}

\setlist[enumerate,1]{label={\upshape(\arabic*)},
ref={\upshape\thetheorem(\arabic*)}}% Default enumerate should be upshape.

\makeatletter
\newcommand{\itemlabel}[1]{\protected@edef\@currentlabel{{(\arabic{enumi})}}\oldlabel{#1}}

\newcommand\crefenumtype{} % Add cref key to lists
\SetEnumitemKey{cref}{%
  before={\def\crefenumtype{#1}%
  \let\oldlabel\label
  \renewcommand{\label}[2][]{%
    \ifx\crefenumtype\empty
      \oldlabel{####2}%
    \else
      \oldlabel[\crefenumtype]{####2}%
      \itemlabel{####2@item}%
    \fi}%
  }%
}
\newcommand{\tref}[1]{\ref{#1@item}}
\makeatother

\DeclareMathAlphabet{\dutchcal}{U}{dutchcal}{b}{n}
\newcommand{\defOp}[3][\mathrm]{%
  \newcommand{#2}{\ensuremath{\mathop{#1{#3}}}\xspace}%
}
\newcommand{\defOrd}[3][\mathrm]{%
  \newcommand{#2}{\ensuremath{\mathord{#1{#3}}}\xspace}%
}

\let\d\relax
\defOrd[\dutchcal]{\d}{d}
\defOrd{\downset}{\downarrow}
\defOrd{\upset}{\uparrow}

\defOrd[\mathfrak]{\B}{B}
\defOrd{\K}{K}
\let\L\relax
\defOrd[\mathsf]{\L}{L}
\defOrd{\N}{N}

\defOp{\spec}{spec}
\defOp{\pt}{pt}
\defOp{\alg}{alg}
\defOp{\con}{con}
\defOp{\reg}{reg}
\defOp{\cen}{zer}
\defOp{\clopsup}{ClopSUp}

\defOp{\opup}{OpUp}
\defOp{\clup}{ClUp}
\defOp{\SFilt}{SFilt}
\defOp{\Filt}{Filt}
\defOp{\SUp}{SUp}
\defOp{\clopup}{ClopUp}
\defOp{\loc}{loc}
\defOp{\cl}{cl}
\defOp{\intr}{int}

\newcommand\twoheaduparrow{\mathord{\rotatebox[origin=c]{90}{$\twoheadrightarrow$}}}

\newcommand{\cat}[1]{\mathbf{#1}}
\newcommand{\defCat}[2]{\defOp[\cat]{#1}{#2}}

\defCat{\DLat}{DLat}
\defCat{\Frm}{Frm}
\defCat{\SFrm}{SFrm}
\defCat{\AlgFrm}{AlgFrm}
\defCat{\AriFrm}{AriFrm}
\defCat{\CohFrm}{CohFrm}
\defCat{\LStoneFrm}{LStoneFrm}
\defCat{\StoneFrm}{StoneFrm}
\defCat{\ConFrm}{ConFrm}
\defCat{\StConFrm}{StConFrm}
\defCat{\StKFrm}{StKFrm}
\defCat{\ConRFrm}{ConRFrm}
\defCat{\KRFrm}{KRFrm}
\defCat{\Top}{Top}
\defCat{\LCSob}{LKSob}
\defCat{\SLCSp}{StLKSp}
\defCat{\SKSp}{StKSp}
\defCat{\KHaus}{KHaus}
\defCat{\LCHaus}{LKHaus}
\defCat{\Sob}{Sob}
\defCat{\KBSob}{KBSob}
\defCat{\SKBSp}{SKBSp}
\defCat{\Spec}{Spec}
\defCat{\LStone}{LStone}
\defCat{\Stone}{Stone}
\defCat{\Pries}{Pries}
\defCat{\LPries}{LPries}
\defCat{\AlgLPries}{AlgL}
\defCat{\AriLPries}{AriLP}
\defCat{\CohLPries}{CohLP}
\defCat{\StoneLPries}{StoneLP}
\defCat{\LStoneLPries}{LStoneLP}
\defCat{\ConLPries}{ConLP}
\defCat{\StConLPries}{StConLP}
\defCat{\StKLPries}{StKLP}
\defCat{\ConRLPries}{ConRLP}
\defCat{\KRLPries}{KRLP}

\definecolor{plotd}{HTML}{000000}
\definecolor{plotc}{HTML}{525252}
\definecolor{plotb}{HTML}{969696}
\definecolor{plota}{HTML}{d9d9d9}
\definecolor{cyand}{HTML}{08519c}
\definecolor{cyanc}{HTML}{2171b5}
\definecolor{cyanb}{HTML}{4292c6}
\definecolor{cyana}{HTML}{9ecae1}

\tikzcdset{dual/.style={squiggly}}

\newcommand{\filter}[1]{\dutchcal{#1}}